\documentclass[a4paper,10pt]{article}
\usepackage{lmodern}
\usepackage[utf8]{inputenc}
\usepackage[T1]{fontenc}
\usepackage{geometry}
\usepackage{microtype}
\usepackage{authblk}
\usepackage{amsmath,amssymb}
\usepackage{amsthm}
\usepackage{mathtools}
\usepackage{graphicx}
\usepackage{xcolor}
\usepackage{subcaption}
\usepackage{hyperref}
\hypersetup{
    colorlinks,
    linkcolor={red!50!black},
    citecolor={blue!50!black},
    urlcolor={blue!80!black}
}

\theoremstyle{definition}
\newtheorem{thm}{Theorem}[section]
\newtheorem{defn}[thm]{Definition}
\newtheorem{lem}[thm]{Lemma}

\newtheorem*{commr}{\textcolor{red}{Comment}}
\newtheorem*{commb}{\textcolor{blue}{To do}}

\newcommand{\mbfb}{\mathbf b}
\newcommand{\mbfc}{\mathbf c}

\newcommand{\mbfy}{\mathbf y}
\newcommand{\mbfA}{\mathbf A}

\newcommand{\mbfD}{\mathbf D}
\newcommand{\mbfH}{\mathbf H}

\newcommand{\mbfM}{\mathbf M}
\newcommand{\mbfP}{\mathbf P}

\renewcommand{\O}{\mathcal O}

\newcommand{\abs}[1]{\lvert #1\rvert}

\DeclareMathOperator*{\eq}{=}

\DeclareFontFamily{U}{mathx}{\hyphenchar\font45}
\DeclareFontShape{U}{mathx}{m}{n}{<-> mathx10}{}
\DeclareSymbolFont{mathx}{U}{mathx}{m}{n}
\DeclareMathAccent{\widebar}{0}{mathx}{"73}
\newcommand{\PPDS}{\widehat P}
\newcommand{\DPDS}{\widehat D}
\newcommand{\pPDS}{\widehat p}
\newcommand{\dPDS}{\widehat d}

\title{Unconditionally positive and conservative third order modified Patankar-Runge-Kutta discretizations of production-destruction systems}
\author[a]{Stefan Kopecz}
\author[a]{Andreas Meister}
\affil[a]{Institute of Mathematics, University of Kassel, Kassel, Germany}
\date{\today}

\begin{document}
\maketitle
\abstract{Modified Patankar-Runge-Kutta (MPRK) schemes are numerical methods for the solution of positive and conservative production-destruction systems. They adapt explicit Runge-Kutta schemes in a way to ensure positivity and conservation irrespective of the time step size.

The first two members of this class, the first order MPE scheme and the second order MPRK22(1) scheme, were introduced in \cite{BDM2003} and have been successfully applied in a large number of applications. Recently, we introduced a general definition of MPRK schemes and presented a thorough investigation of first and second order MPRK schemes in \cite{KopeczMeister2017}.

A potentially third order Patankar-type method was introduced in \cite{FormaggiaScotti2011}.
This method uses the MPRK22(1) scheme of \cite{BDM2003} as a predictor and a modification of the BDF(3) multistep method as a corrector. 
It restricts to the MPRK22(1) approximation, whenever the positivity of the corrector cannot be guaranteed.
Hence, this method is at most third order accurate and at least second order accurate.

In this paper we continue the work of \cite{KopeczMeister2017} and present necessary and sufficient conditions for third order MPRK schemes. For the first time, we introduce MPRK schemes, which are third order accurate independent of the specific positive and conservative system under consideration. The theoretical results derived within the first part are subsequently confirmed by numerical experiments for the entire domain of linear and nonlinear as well as nonstiff and stiff systems of differential equations.}
\section{Introduction}
A wide variety of mathematical models for real life problems are given in the form of a system of partial differential equations including stiff production-destruction terms. The development of numerical methods is therefore often based on splitting approaches, where the discretization of the convection and diffusion terms is conducted within a first step and the approximation of the source terms is realized  subsequently. Thereby, the time step size of the first part should also be applicable within the second step and even particular properties like conservativity of the source terms and positivity of the constituents have to be maintained independent of the time step size in order to obtain a reliable, appropriate and efficient simulation. Whereas finite volume schemes are well established for the discretization of convection-diffusion equations, the development of unconditionally positivity preserving and conservative methods of higher order for stiff systems of ordinary differential equations is still a challenge. To overcome this gap, the paper is devoted to the derivation and investigation of a class of third order modified Patankar schemes. Therefore, we consider production-destruction systems (PDS) of the form
\begin{align}\label{eq:pds}
 \frac{d y_i}{dt}(t)=P_i(\mbfy(t))-D_i(\mbfy(t)),\quad i=1,\dots,N.
\end{align}
By $\mbfy=(y_1,\dotsc, y_N)^T$ we denote the vector of constituents, which depends on time $t$. Both, the production terms $P_i$ and the destruction terms $D_i$ are assumed to be non-negative, that is $P_i, D_i\geq 0$ for $i=1,\dots,N$.
Furthermore, the production and destruction terms can be written as
\begin{align}\label{eq:pijdij}
 P_i(\mbfy) = \sum_{j=1}^N p_{ij}(\mbfy),\quad D_i(\mbfy) = \sum_{j=1}^N d_{ij}(\mbfy),
\end{align}
where $d_{ij}(\mbfy)\geq 0$ is the rate at which the $i$th constituent transforms into the $j$th component, while $p_{ij}(\mbfy)\geq 0$ is the rate at which the $j$th constituent transforms into the $i$th component.

We are interested in PDS which are positive as well as fully conservative.
\begin{defn}
The PDS \eqref{eq:pds} is called \textit{positive}, if positive initial values,  $y_i(0)> 0$ for $i=1,\dots,N$, imply positive solutions, $y_i(t)>0$ for $i=1,\dots,N$, for all times $t>0$.
\end{defn}
\begin{defn}\label{defn:consode}
The PDS \eqref{eq:pds}, \eqref{eq:pijdij} is called \textit{conservative}, if for all $i,j=1,\dots,N$ and $\mbfy\geq 0$, we have
$
p_{ij}(\mbfy) = d_{ji}(\mbfy). 
$
The system is called \textit{fully conservative}, if in addition
$
p_{ii}(\mbfy)=d_{ii}(\mbfy)=0
$
holds for all $\mbfy\geq 0$ and $i=1,\dots,N$.
\end{defn}
In the following, we will assume that the PDS \eqref{eq:pds} is fully conservative, since every conservative PDS can be rewritten as an equivalent fully conservative PDS.
For a fully conservative PDS, \eqref{eq:pijdij} can be written as
\[
P_i(\mbfy) = \sum_{\substack{j=1\\j\ne i }}^N p_{ij}(\mbfy),\quad D_i(\mbfy) = \sum_{\substack{j=1\\j\ne i }}^N d_{ij}(\mbfy).
\]
But for the sake of a simple notation, we will always use the form \eqref{eq:pijdij}.
Examples of positive and conservative PDS, which model academic as well as realistic applications, can be found in Section~\ref{sec:testcases}.

If a PDS is conservative the sum of its constituents $\sum_{i=1}^N y_i(t)$ remains constant in time, since we have
\[
\frac{d}{dt}\sum_{i=1}^N y_i=\sum_{i=1}^N \bigl(P_i(\mbfy)-D_i(\mbfy)\bigr) = \sum_{i,j=1}^N \bigl(p_{ij}(\mbfy)-d_{ij}(\mbfy)\bigr) = \sum_{i,j=1}^N \bigl(\underbrace{p_{ij}(\mbfy)-d_{ji}(\mbfy)}_{=0}\bigr) = 0.
\]
This motivates the definition of a conservative numerical scheme.
\begin{defn}
Let $\mbfy^n$ denote an approximation of $\mbfy(t^n)$ at time level $t^n$. 
 The one-step method
 \[
 \mbfy^{n+1} = \mbfy^n + \Delta t \Phi(t^n,\mbfy^n,\mbfy^{n+1},\Delta t)
 \]
 is called
 \begin{itemize}
 \item \textit{unconditionally conservative}, if 
 \[
 \sum_{i=1}^N \left(y_i^{n+1}-y_i^n\right)=0
 \]
 is satisfied for all $n\in\mathbb N$ and $\Delta t>0$.
 \item \textit{unconditionally positive}, if it guarantees $\mbfy^{n+1}>0$ for all $\Delta t>0$ and $\mbfy^n>0$.
 \end{itemize}
\end{defn}

An explicit $s$-stage Runge-Kutta method for the solution of an ordinary differential equation
$y'(t) = f(t,y(t))$
is given by
\begin{equation*}
\begin{split}
 y^{(k)} &= y^n + \Delta t\sum_{\nu=1}^{k-1} a_{k\nu} f(t^n+c_\nu \Delta t,y^{(\nu)}),\quad k=1,\dots,s,\\
 y^{n+1}&=y^n+\Delta t \sum_{k=1}^s b_k 
 f(t^n+c_k\Delta t,y^{(k)}).
\end{split}
\end{equation*}
The method is characterized by its coefficients $a_{k\nu}$, $b_k$, $c_k$ for $k=1,\dots,s$, $\nu=1,\dots,k-1$ and can be represented 
by the Butcher tableau
\[
\begin{array}{c|c}
\mbfc        &    \mbfA\\\hline
          & \mbfb
\end{array},
\]
with $\mbfA = (a_{k\nu})_{k,\nu=1,\dots,s}$, $\mbfc=(c_1,\dots,c_s)^T$ and $\mbfb=(b_1,\dots,b_s)$.
Applied to \eqref{eq:pds} the method reads
\begin{subequations}
\begin{align}\label{rkscheme}
 y_i^{(k)} &= y_i^n + \Delta t\sum_{\nu=1}^{k-1} a_{k\nu} 
 \sum_{j=1}^N \left(p_{ij}(\mbfy^{(k)})-d_{ij}(\mbfy^{(k)})\right),\quad k=1,\dots,s,\\ \label{eq:rkschemenp1}
 y_i^{n+1}&=y_i^n+\Delta t \sum_{k=1}^s b_k 
 \sum_{j=1}^N\left( p_{ij}(\mbfy^{(k)})-d_{ij}(\mbfy^{(k)})\right).
\end{align}
\end{subequations}

The idea of the modified Patankar-Runge-Kutta (MPRK) schemes is to adapt explicit Runge-Kutta schemes in such a way that they become positive irrespective of the chosen time step size $\Delta t$, while still maintaining their inherent property to be conservative. 
One approach to achieve unconditional positivity is the so-called Patankar-trick introduced in \cite{Patankar1980} as \textit{source term linearization} in the context of turbulent flow. If we modify \eqref{eq:rkschemenp1} and add a weighting of the destruction terms like
\begin{equation*}
y_i^{n+1} = y_i^n + \Delta t\sum_{k=1}^{s} b_{k} 
 \sum_{j=1}^N\biggl( p_{ij}(\mbfy^{(k)})-d_{ij}(\mbfy^{(k)})\frac{y_i^{n+1}}{\sigma_i}\biggr),
\end{equation*} 
we obtain
\begin{equation*}
y_i^{n+1} = \frac{y_i^n + \Delta t\sum_{k=1}^s b_k 
 \sum_{j=1}^N p_{ij}(\mbfy^{(k)})}{1 + \Delta t\sum_{k=1}^s b_k 
 \sum_{j=1}^N d_{ij}(\mbfy^{(k)})/\sigma_i}.
\end{equation*} 
Thus, if $y_i^n$, the weights $b_k$ for $k=1,\dots,s$ and $\sigma_i$ are positive, so is $y_i^{n+1}$. 
The crucial idea of the Patankar-trick is to multiply the destruction terms with weights that comprise $y_i^{n+1}$ as a factor themselves.

Weighting only the destruction terms will result in a non-conservative scheme. So the production terms have to be weighted accordingly as well. Since we have $d_{ij}(\mbfy)=p_{ji}(\mbfy)$, the proper weight for $p_{ij}(\mbfy^{(k)})$ is $y_j^{n+1}/\sigma_j$.
\begin{defn}\label{def:MPRKdefn}
Given a non-negative Runge-Kutta matrix $\mbfA=(a_{ij})_{i,j=1,\dots,s}$, non-negative weights $b_1,\dotsc,b_s$ and $\delta\in\{0,1\}$, the scheme
\begin{subequations}\label{eq:MPRK}
\begin{align}\label{eq:MPRKstages}
y_i^{(k)} &= y_i^n + \Delta t\sum_{\nu=1}^{k-1} a_{k\nu} 
 \sum_{j=1}^N \biggl(p_{ij}(\mbfy^{(\nu)})(1-\delta)+p_{ij}(\mbfy^{(\nu)})\frac{y_j^{(k)}}{\pi^{(k)}_j}\delta-d_{ij}(\mbfy^{(\nu)})\frac{y_i^{(k)}}{\pi^{(k)}_i}\biggr),\quad k=1,\dots,s,\\\label{eq:MPRKapprox}
 y_i^{n+1}&=y_i^n+\Delta t \sum_{k=1}^s b_k 
 \sum_{j=1}^N\biggl( p_{ij}(\mbfy^{(k)})\frac{y_j^{n+1}}{\sigma_j}-d_{ij}(\mbfy^{(k)})\frac{y_i^{n+1}}{\sigma_i}\biggr),
\end{align}
\end{subequations}
for $i=1\dots,N$, is called \textit{modified Patankar-Runge-Kutta scheme} (MPRK) if
\begin{enumerate}
 \item $\pi_i^{(k)}$ and $\sigma_i$ are unconditionally positive for $k=1,\dots, s$ and $i=1,\dots,N$, 
 \item $\pi_i^{(k)}$ is independent of $y_i^{(k)}$ and $\sigma_i$ is independent of $y_i^{n+1}$ for $k=1,\dots, s$ and $i=1,\dots,N$.
\end{enumerate}
The weights $1/\sigma_i$ and $1/\pi_i^{(k)}$ are called Patankar-weights and the denominators $\sigma_i$ and $\pi_i^{(k)}$ are called Patankar-weight denominators (PWD).
\end{defn}

Due to the introduction of the Patankar-weights, $s$ linear systems of size $N\times N$ need to be solved to obtain the stage values and the approximation at the next time level. 
In consideration of $p_{ii}=d_{ii}=0$ for $i=1,\dots,N$, the scheme \eqref{eq:MPRK} can be written in matrix-vector notation as
\begin{subequations}\label{eq:MPRKMVs}
\begin{align}\label{eq:MPRKMVstage}
 \mbfM^{(k)}\mbfy^{(k)}&=\mbfy^n+(1-\delta)\Delta t\mbfP(\mbfy^n),\quad k=1,\dots,s,\\\label{eq:MPRKMV}
 \mbfM\mbfy^{n+1}&=\mbfy^n,
\end{align}
\end{subequations}
with $\mbfP(\mbfy^n)=(P_1(\mbfy^n),\dots,P_N(\mbfy^n))^T$ and 
 \begin{equation}\label{eq:MPRKmatvecstage}
 \begin{split}
	m_{ii}^{(k)} &= 1 + \Delta t\sum_{\nu=1}^{k-1} a_{k\nu}
 \sum_{j=1}^N d_{ij}(\mbfy^{(\nu)})/\pi_i^{(k)} >0,\quad i=1,\dots, N,\\
  m_{ij}^{(k)} &=-\Delta t\delta\sum_{\nu=1}^{k-1} a_{k\nu}  p_{ij}(\mbfy^{(\nu)})/\pi_j^{(k)} \leq 0,\quad i,j=1,\dots,N,\,i\ne j,
  \end{split}
 \end{equation}
 for $k=1,\dots,s$ and
  \begin{equation}\label{MPRKmatvec}
   \begin{split}
  m_{ii} &= 1 + \Delta t\sum_{k=1}^s b_k 
 \sum_{j=1}^N d_{ij}(\mbfy^{(k)})/\sigma_i >0,\quad i=1,\dots,N,\\
  m_{ij} &=-\Delta t\sum_{k=1}^s b_k  p_{ij}(\mbfy^{(k)})/\sigma_j \leq 0,\quad i,j=1,\dots,N,\,i\ne j.
  \end{split}
 \end{equation}
If $\delta = 0$, the matrices $\mbfM^{(k)}$ become diagonal and the production terms appear on the right hand side of \eqref{eq:MPRKMVstage}.
 
The following two lemmas of \cite{KopeczMeister2017} show that MPRK schemes, as defined in Definition~\ref{def:MPRKdefn}, are indeed unconditionally positive and conservative. 
\begin{lem}\label{lem:MPRKcons}
 A MPRK scheme \eqref{eq:MPRK} applied to a conservative PDS is unconditionally conservative. 
 If $\delta=1$, the same holds for all stage values, this is
 $\sum_{i=1}^N (y_i^{(k)}-y_i^n)=0$
 for $k=1,\dots,s$.
\end{lem}

\begin{lem}\label{lem:MPRKpos}
A MPRK scheme \eqref{eq:MPRK} is unconditionally positive. 
 The same holds for all the stages of the scheme, this is for all $\Delta t>0$ and $\mbfy^n>0$ we have $\mbfy^{(k)}>0$ for $k=1,\dots,s$.
\end{lem}

Lemmas~\ref{lem:MPRKcons} and \ref{lem:MPRKpos} show that the MPRK schemes as defined in Definition~\ref{def:MPRKdefn} possess the desired properties of unconditional positivity and conservation. 
The only quantities left to choose are the PWDs $\sigma_i$ and $\pi_i^{(k)}$ for $i=1,\dots,N$ and $k=1,\dots,s$. 
In \cite{KopeczMeister2017} we introduced the second order MPRK22($a_{21}$) schemes, which use $\pi_i=y_i^n$ and $\sigma_i=y_i^n(y_i^{(2)}/y_i^n)^{1/a_{21}}$ for $i=1,\dots,N$.

The MPRK22(1) scheme is equivalent to the original MPRK scheme introduced in \cite{BDM2003}. This scheme and the first order modified Patankar-Euler scheme of \cite{BDM2003} have been successfully applied to solve physical, biogeochemical and ecosystem models (\cite{BDM2005,BBKMNU2006,BMZ2009,HenseBurchard2010,HenseBeckmann2010,MeisterBenz2010,WHK2013,SemeniukDastoor2017}), and have also proven beneficial in astrophysics \cite{KlarMuecket2010,Gressel2016}. 

In \cite{SchippmannBurchard2011} it was demonstrated that the MPRK22(1) scheme outperforms standard Runge-Kutta and Rosenbrock methods when solving biogeochemical models without multiple
source compounds per system reactions. 
The same was shown with respect to workload in \cite{BonaventuraDellaRocca2016}, where the Brusselator PDS was solved with different time integration schemes.

In \cite{BBKS2007,BRBM2008} second order schemes, which ensure conservation in a biochemical sense, were introduced. 
These schemes require the solution of a non-linear equation in each time step. 
Other schemes for the same purpose were recently presented in \cite{RadtkeBurchard2015}. 
These explicit schemes incorporate the MPRK schemes of \cite{BDM2003} to achieve multi-element conservation for stiff problems.
A potentially third order Patankar-type scheme was introduced in \cite{FormaggiaScotti2011}.
This method uses the MPRK22(1) scheme a as predictor and a modification of the BDF(3) multistep method as a corrector. 
It yields the MPRK22(1) approximation, whenever the positivity of the corrector approximation cannot be guaranteed.

Modified Patankar-Runge-Kutta type schemes are also used in the context of partial differential equations. An implicit first order Patankar-type scheme based on a third order SDIRK method was presented in \cite{MeisterOrtleb2014} and applied to the shallow water equations.
In \cite{OrtlebHundsdorfer2016} Patankar-type Runge-Kutta schemes for linear PDEs were investigated. 

In the present paper, we extend the work of \cite{KopeczMeister2017} to third order. We present necessary and sufficient conditions for third order three-stage MPRK schemes, and introduce two families of third order MPRK methods.
To our knowledge, this is the first time that third order Patankar-type schemes are presented.
 
The paper is organized as follows. Section~\ref{sec:MPRKorder} deals with the derivation of conditions for third order three-stage MPRK schemes. In this section also novel third order MPRK schemes are introduced. The test problems of Section~\ref{sec:testcases} are used in Section~\ref{sec:numres} to show numerical experiments with these novel schemes.
\section{Third order MPRK schemes}\label{sec:MPRKorder}
In this section we assume that all occurring PDS are positive. 
To prove convergence of the MPRK schemes we investigate the local truncation errors.
In doing so we make frequent use of the Landau symbol $\O$ and omit to specify the limit process $\Delta t\to 0$ each time.
As customary, we identify $y_i^n$ and $y_i(t^n)$ for $i=1,\dots,N$ when studying the truncation errors. 
Furthermore, since we are dealing with positive PDS we assume $y_i^n>0$ for $i=1,\dots,N$.

To derive necessary conditions that guarantee a certain order of an MPRK scheme, it suffices to consider specific PDS.
In this regard, the following family of PDS will be very helpful. 
Given parameters $I,J\in\{1,\dots,N\}$, $I\ne J$, $\mu>0$ and $\kappa\in\{1,2\}$,
we consider
\begin{subequations}\label{eq:PDSorder}
\begin{equation}
 \frac{dy_i}{dt}(t) = \PPDS_i(\mbfy(t)) - \DPDS_i(\mbfy(t)),\quad i=1,\dots,N,
\end{equation}
with
\begin{equation}
\PPDS_i(\mbfy) = \begin{cases}\mu y_I^\kappa, & i=J,\\ \phantom{y}0, &\text{otherwise},\end{cases}
\quad\text{and}\quad
\DPDS_i(\mbfy) = \begin{cases}\mu y_I^\kappa, & i=I,\\ \phantom{y}0, &\text{otherwise},\end{cases}
\end{equation}
\end{subequations}
 and initial values $y_i(0)=1$ for $i=1,\dots,N$.
The PDS can be written in the form
\begin{align*}
\frac{dy_I}{dt} =-\mu y_I^\kappa,\qquad
\frac{dy_J}{dt} =\mu y_I^\kappa,\qquad
\frac{dy_i}{dt} = 0, \quad i\in\{1,\dots,N\}\setminus\{I,J\}.
\end{align*}
For $\kappa=1$ the exact solution is given by
\begin{align*}
y_I(t) &= e^{-\mu t},\\
y_J(t) &= 2-e^{-\mu t},\\
y_i(t) &=1, \quad i\in\{1,\dots,N\}\setminus\{I,J\},
\end{align*}
and for $\kappa=2$ the exact solution is given by
\begin{align*}
y_I(t) &= (1+t\mu)^{-1},\\
y_J(t) &= 2-(1+t\mu)^{-1},\\
y_i(t) &=1, \quad i\in\{1,\dots,N\}\setminus\{I,J\}.
\end{align*}
This shows that the PDS is positive.
Writing
\[
\PPDS_i(\mbfy) = \sum_{j=1}^N \pPDS_{ij}(\mbfy),\quad \DPDS_i(\mbfy) = \sum_{j=1}^N \dPDS_{ij}(\mbfy),
\]
with
\[
\pPDS_{ij}(\mbfy)=\begin{cases}\mu y_I^\kappa,& i=J\text{ and }j=I,\\\phantom{y}0,&\text{otherwise},\end{cases}\quad
\dPDS_{ij}(\mbfy)=\begin{cases}\mu y_I^\kappa,& i=I\text{ and }j=J,\\\phantom{y}0,&\text{otherwise},\end{cases}
\]
we see that the PDS is also fully conservative.

The following lemmas are helpful to find necessary and sufficient conditions for the PWDs of third order MPRK schemes.
The first one ensures the boundedness of the MPRK approximations and the stage values.
\begin{lem}\label{lem:mbound}
 Let $\mbfM$, $\mbfM^{(k)}$ be given by \eqref{MPRKmatvec},
 \eqref{eq:MPRKmatvecstage} with $\delta = 1$ and $\mbfM^{-1}=(\widetilde m_{ij})$,  \mbox{$(\mbfM^{(k)})^{-1}=\widetilde m_{ij}^{(k)}$}. Then we have
 \[0\leq \widetilde m_{ij},\widetilde m_{ij}^{(k)}\leq 1,\quad i,j=1,\dots,N,\]
 for $k=1,\dots,s$.
 \end{lem}
\begin{proof}
 See \cite{KopeczMeister2017}.
\end{proof}

The next lemma is useful to separate complicated conditions for the PWDs into simpler ones.
\begin{lem}\label{lem:lemO}
The identity
\begin{equation}\label{eq:lemOaux1}
 \xi_0(\Delta t) + \mu \xi_1(\Delta t) + \mu^2 \xi_2(\Delta t) + \dots+\mu^n \xi_n(\Delta t)=\O(\Delta t^s)\text{ for all }\mu>0
\end{equation}
is equivalent to 
\begin{equation}\label{eq:lemOaux2}
 \xi_i(\Delta t)=\O(\Delta t^s),\quad i=0,\dots,N.
\end{equation}

\end{lem}
\begin{proof}
 Let $0<\mu_0<\mu_1<\dots<\mu_n$. Since \eqref{eq:lemOaux1} is valid for all $\mu>0$, we have
 \begin{equation*}
  \xi_0(\Delta t) + \mu_i \xi_1(\Delta t) + \mu_i^2 \xi_2(\Delta t) + \dots+\mu_i^n \xi_n(\Delta t)=\O(\Delta t^s)
 \end{equation*}
 for $i=0,\dots,n$. This can be rewritten as
 \begin{equation}\label{eq:lemOaux3}
\mathbf V\boldsymbol\xi(\Delta t)=\O(\Delta t^s),
\end{equation}
  with
  \begin{equation*}
   \mathbf V = \begin{pmatrix}1 & \mu_0 & \mu_0^2& \cdots & \mu_0^n\\\vdots &\vdots&\vdots&\ddots&\vdots\\1 & \mu_n & \mu_n^2& \cdots & \mu_n^n\end{pmatrix},\quad \boldsymbol \xi(\Delta t)=(\xi_0(\Delta t),\dots,\xi_n(\Delta t))^T.
  \end{equation*} 
Since all $\mu_i$ are distinct for $i=0,\dots,n$, $\mathbf V$ is a Vandermonde matrix, and hence, regular.
Multiplication of \eqref{eq:lemOaux3} with $\mathbf V^{-1}$ yields
\begin{equation*}
 \boldsymbol\xi(\Delta t)=\mathbf V^{-1}\O(\Delta t^s)=\O(\Delta t^s),
\end{equation*}
Hence, \eqref{eq:lemOaux2} is satisfied.
On the other hand, if \eqref{eq:lemOaux2} is satisfied, so is \eqref{eq:lemOaux1}.
\end{proof}

As a three-stage MPRK scheme is build on an explicit three-stage Runge-Kutta scheme with non-negative parameters, we must characterize these schemes in some way. It is well known, that an explicit three-stage Runge-Kutta scheme
 \begin{equation*}
  \begin{array}{c|ccc}
   0\\
   c_2 & a_{21}\\
   c_3 & a_{31} & a_{32}\\\hline
   & b_1 & b_2 & b_3
  \end{array}
 \end{equation*}
 is third order accurate, if the conditions 
 \begin{subequations}\label{eq:RKorder0}
 \begin{align}
  a_{21}&=c_{2},\label{eq:RKorder1}\\ 
  a_{31}+a_{32}&=c_3,\label{eq:RKorder2}\\
  b_1+b_2+b_3&=1,\label{eq:RKorder3}\\
  b_2c_2+b_3c_3&=1/2,\label{eq:RKorder4}\\ 
  b_2c_2^2+b_3c_3^2&=1/3,\label{eq:RKorder5}\\
  a_{21}a_{32}b_3&=1/6\label{eq:RKorder6}
 \end{align}
 \end{subequations}
are satisfied. 
The last condition particularly implies $a_{21},a_{32},b_3\ne 0$. 

The following lemma shows, that all explicit three-stage Runge-Kutta schemes of order three can be parameterized by families with at most two free parameters. 
Later we will use this lemma to characterize all explicit three-stage Runge-Kutta schemes of order three with non-negative parameters. 
\begin{lem}\label{lem:RK3}

 All explicit third order Runge-Kutta schemes can be parameterized with at most two parameters. The following three cases can occur:
 \paragraph{Case I:}
 \begin{equation*}
 \renewcommand{\arraystretch}{2}
 \setlength{\arraycolsep}{10pt}
  \begin{array}{c|ccc}
   0\\
   \alpha & \alpha\\
   \beta & \dfrac{3\alpha\beta(1-\alpha)-\beta^2}{\alpha(2-3\alpha)} & \dfrac{\beta(\beta-\alpha)}{\alpha(2-3\alpha)}\\\hline
   & 1+\dfrac{2-3(\alpha+\beta)}{6\alpha\beta} & \dfrac{3\beta-2}{6\alpha(\beta-\alpha)} & \dfrac{2-3\alpha}{6\beta(\beta-\alpha)}
  \end{array}
 \end{equation*}
with $\alpha,\beta\ne0$, $\alpha\ne\beta$, $\alpha\ne\frac23$.
\paragraph{Case II:} 
 \begin{equation*}
 \renewcommand{\arraystretch}{2}
 \setlength{\arraycolsep}{5pt}
  \begin{array}{c|ccc}
   0\\
   \dfrac23 & \dfrac23\\
  \dfrac23 & \dfrac23-\dfrac{1}{4\gamma} & \dfrac{1}{4\gamma}\\\hline
   & \dfrac14 & \dfrac34-\gamma & \gamma
  \end{array}
 \end{equation*}
 with $\gamma\ne 0$.
 \paragraph{Case III:} 
 \begin{equation*}
 \renewcommand{\arraystretch}{2}
 \setlength{\arraycolsep}{5pt}
  \begin{array}{c|ccc}
   0\\
   \dfrac23 & \dfrac23\\
   0 & -\dfrac{1}{4\gamma} & \dfrac{1}{4\gamma}\\\hline
   & \dfrac14-\gamma & \dfrac34 & \gamma
  \end{array}
 \end{equation*}
 with $\gamma\ne 0$.
\end{lem}
\begin{proof}
 See \cite{SWP2012,RalstonRabinowitz2001}.
\end{proof}

To build an MPRK scheme based on an explicit third order three-stage Runge-Kutta scheme, we must ensure the non-negativity of the occurring parameters. The following lemma characterizes all such Runge-Kutta schemes.
\begin{lem}\label{lem:RK3pos}
 All explicit three-stage third order Runge-Kutta schemes with non-negative parameters can be represented by the following Butcher tableaus:
  \paragraph{Case I:} 
 \begin{equation*}
 \renewcommand{\arraystretch}{2}
 \setlength{\arraycolsep}{10pt}
  \begin{array}{c|ccc}
   0\\
   \alpha & \alpha\\
   \beta & \dfrac{3\alpha\beta(1-\alpha)-\beta^2}{\alpha(2-3\alpha)} & \dfrac{\beta(\beta-\alpha)}{\alpha(2-3\alpha)}\\\hline
   & 1+\dfrac{2-3(\alpha+\beta)}{6\alpha\beta} & \dfrac{3\beta-2}{6\alpha(\beta-\alpha)} & \dfrac{2-3\alpha}{6\beta(\beta-\alpha)}
  \end{array}
 \end{equation*}
with 
\begin{equation}\label{eq:RK3poscond}
\left.\begin{matrix}2/3\leq \beta\leq 3\alpha(1-\alpha)\\
3\alpha(1-\alpha)\leq \beta\leq 2/3\\
({3\alpha-2})/({6\alpha-3})\leq\beta\leq 2/3
\end{matrix}\right\}
\quad\text{for}\quad
\left\{\begin{matrix}
1/3\leq\alpha <\frac23,\\
2/3<\alpha<\alpha_0,\\
\alpha>\alpha_0,
\end{matrix}\right.
\end{equation}
and
\begin{equation*}
\alpha_0=\frac16\left(3+(3-2\sqrt 2)^{1/3}+(3+2\sqrt 2)^{1/3}\right)\approx 0.89255.
\end{equation*}
\paragraph{Case II:}
 \begin{equation*}
 \renewcommand{\arraystretch}{2}
 \setlength{\arraycolsep}{5pt}
  \begin{array}{c|ccc}
   0\\
   \dfrac23 & \dfrac23\\
  \dfrac23 & \dfrac23-\dfrac{1}{4\gamma} & \dfrac{1}{4\gamma}\\\hline
   & \dfrac14 & \dfrac34-\gamma & \gamma
  \end{array}
 \end{equation*}
 with \[\frac38\leq\gamma\leq \frac34.\]
\end{lem}
\begin{proof}
 According to Lemma~\ref{lem:RK3} we have to distinguish between three different cases.
 
 In case III we have $b_3=\gamma$ and $a_{31}=-\frac1{4\gamma}$.
 Thus, $b_3>0$ implies $a_{31}<0$ and hence, negative Runge-Kutta parameters are inevitable in this case.
 
 In case II we must restrict $b_3=\gamma>0$, such that $a_{31}=\frac23-\frac1{4\gamma}\geq 0$ and $b_2=\frac{3}{4}-\gamma\geq 0$.
 This is the case for $\frac 38\leq \gamma\leq \frac34$.
 
 In case I things become more technical. 
 First, we need to ensure $a_{21}=\alpha>0$. 
 Since $\beta\ne0$ and $\beta\ne\alpha$, we must have $a_{32}={\beta(\beta-\alpha)}/{(\alpha(2-3\alpha))}>0$, and due to
 $\beta=a_{31}+a_{32}$, $\beta>0$ must hold as well.
 Next, we present conditions for the non-negativity of the remaining Runge-Kutta parameters subject to $\alpha,\beta>0$.
 From $a_{31}=(3\alpha\beta(1-\alpha)-\beta^2)/(\alpha(2-3\alpha))\geq 0$ we can conclude 
  \begin{align*}
 \left.\begin{matrix}
  0<\beta\leq 3\alpha(1-\alpha)\\
  \beta\geq 3\alpha(1-\alpha)\\
   \beta>0 
 \end{matrix}\right\}\quad \text{for}\quad \left\{\begin{matrix}
 0<\alpha<2/3,\\2/3<\alpha<1,\\\alpha\geq 1.
 \end{matrix}\right.
 \end{align*}
To ensure $a_{32}={\beta(\beta-\alpha)}/{(\alpha(2-3\alpha))}>0$, we must have
  \begin{align*}
 \left.\begin{matrix}
  \beta>\alpha\\
  0<\beta<\alpha 
 \end{matrix}\right\}\quad \text{for}\quad \left\{\begin{matrix}
 0<\alpha<2/3,\\\alpha> 2/3.
 \end{matrix}\right.
 \end{align*}
The requirement $b_1=1+(2-3(\alpha+\beta))/(6\alpha\beta)\geq 0$ demands
  \begin{align*}
 \left.\begin{matrix}
  0<\beta\leq (3\alpha-2)/(6\alpha-3)\\
  \beta>0\\
 \beta\geq (3\alpha-2)/(6\alpha-3) 
 \end{matrix}\right\}\quad \text{for}\quad \left\{\begin{matrix}
 0<\alpha<1/2,\\1/2\leq\alpha<2/3,\\\alpha> 2/3.
 \end{matrix}\right.
 \end{align*}
 To guarantee $b_2=(3\beta-2)/(6\alpha(\beta-\alpha))\geq 0$, the conditions
   \begin{align*}
 \left.\begin{matrix}
  0<\beta<\alpha\text{ or }\beta \geq 2/3\\
 0 < \beta\leq 2/3\text{ or }\beta > \alpha
 \end{matrix}\right\}\quad \text{for}\quad \left\{\begin{matrix}
 0<\alpha<2/3,\\\alpha> 2/3.
 \end{matrix}\right.
 \end{align*}
 are necessary. 
 Finally, $b_3=({2-3\alpha})/({6\beta(\beta-\alpha)})>0$ implies
    \begin{align*}
 \left.\begin{matrix}
  \beta>\alpha\\
 0 < \beta<\alpha
 \end{matrix}\right\}\quad \text{for}\quad \left\{\begin{matrix}
 0<\alpha<2/3,\\\alpha> 2/3.
 \end{matrix}\right.
 \end{align*}
 Merging the above conditions, we obtain \eqref{eq:RK3poscond}, in which $\alpha_0$ denotes the unique solution of $3\alpha(1-\alpha)=(3\alpha-2)/(6\alpha-3)$.
 The region of feasibility, which contains all pairs $(\alpha,\beta)$ that ensure non-negativity of the Runge-Kutta parameters, is shown in Figure~\ref{fig:RK3pos}.
 \begin{figure}
  \centering
  \includegraphics[width=.8\textwidth]{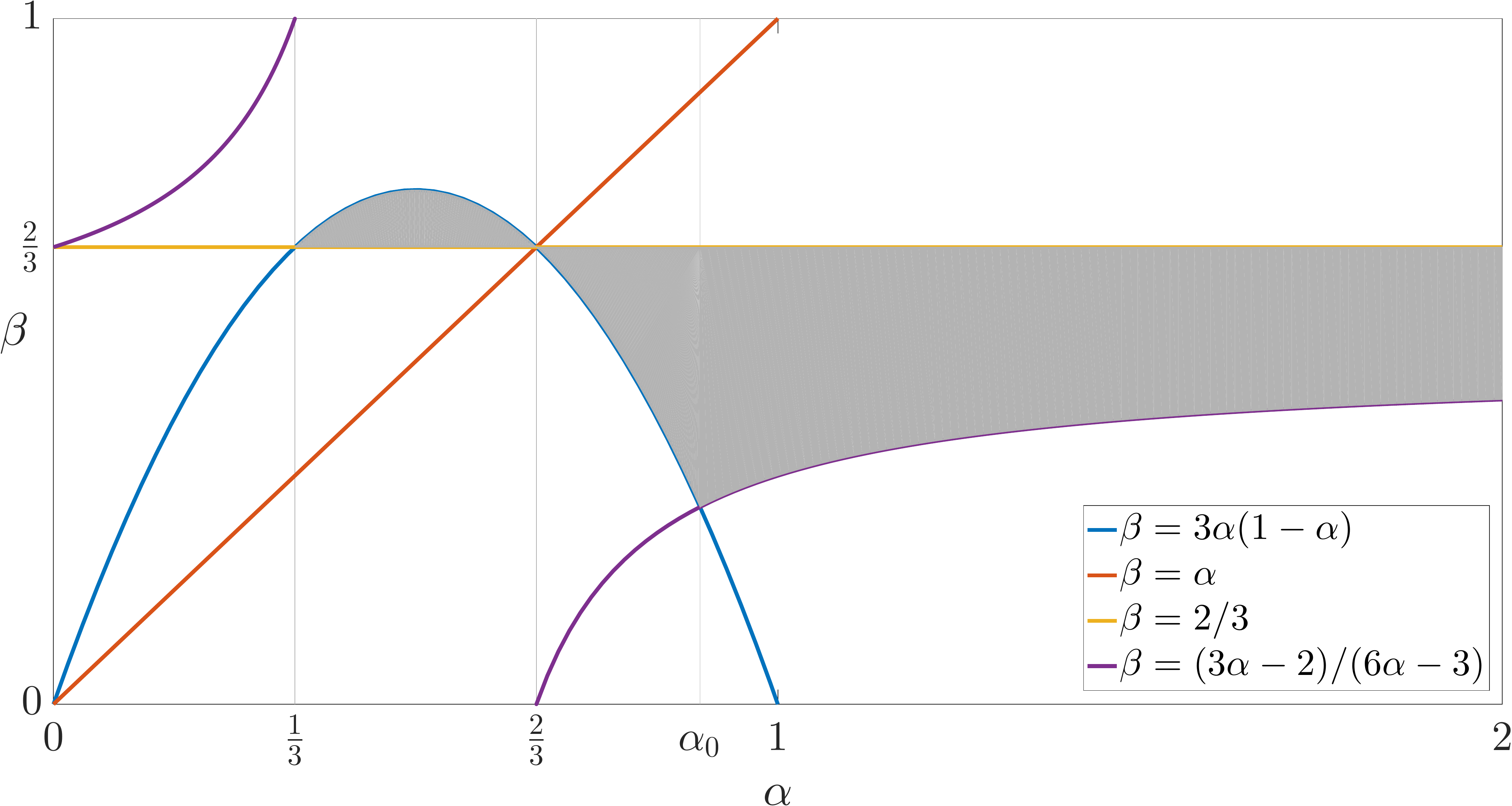}
  \caption{Feasible region (gray) that contains all pairs $(\alpha,\beta)$, for which the Runge-Kutta parameters in case I of Lemma~\ref{lem:RK3} are non-negative. The restriction $0\leq \alpha \leq 2$ was made only to facilitate visualization, the region is unbounded to the left.}
  \label{fig:RK3pos}
 \end{figure}
\end{proof}

To see that $\pi_i=y_i^n+\O(\Delta t)$ and $\rho_i=y_i^n+\O(\Delta t)$ for $i=1,\dots,N$ are necessary conditions for the PWDs of a third order MPRK scheme, the following lemma will be helpful.
\begin{lem}\label{lem:MPRK33sys}
Given an explicit three-stage Runge-Kutta scheme of order three with non-negative parameters,
 the nonlinear system
 \begin{subequations}\label{eq:limsys}
 \begin{align}\label{eq:limsys1}
  b_2a_{21}x+b_3(a_{31}+a_{32})y&=\frac12,\\
  b_2a_{21}^2 x^2+b_3(a_{31}+a_{32})^2 y^2&=\frac13,\label{eq:limsys2}\\
  xy&=1,\label{eq:limsys3}
  \end{align}
 \end{subequations}
has the unique positive solution 
\begin{equation}\label{eq:sol1}
x=y=1.
\end{equation}
\end{lem}
\begin{proof}

First, we note that \eqref{eq:sol1} is a solution of \eqref{eq:limsys}, owing to \eqref{eq:RKorder0}.

Next, we show that no other solutions exist. 
If $b_2= 0$, \eqref{eq:RKorder4} becomes $b_3(a_{31}+a_{32})=1/2$ and hence \eqref{eq:limsys1} reads $y/2=1/2$, which implies $y=1$. 
Similar, owing to \eqref{eq:RKorder5}, $y=1$ is the only positive solution of \eqref{eq:limsys2}, hence, we can conclude $x=1$ from \eqref{eq:limsys3}.
Thus, \eqref{eq:sol1} is the only solution of \eqref{eq:limsys}, if $b_2=0$.

From now on, we assume $b_2\ne0$.
As $b_3(a_{31}+a_{32})\ne 0$ as well, since $a_{32}>0$, $b_3>0$ and $a_{31}\geq 0$, \eqref{eq:limsys1} represents a line and \eqref{eq:limsys2} represents an ellipse in the $x$-$y$-plane.
There are at most two intersections of the line and the ellipse, and thus,  the system \eqref{eq:limsys} has at most two solutions.
We already know that one of them is \eqref{eq:sol1}.
To find the hypothetical other one, we assume $y\ne 1$ and compute the intersection of the line \eqref{eq:limsys1} and the hyperbola \eqref{eq:limsys3}. 
Subtraction of \eqref{eq:RKorder4} from \eqref{eq:limsys1} yields
\begin{equation*}
 b_2 a_{21} (x-1)+b_3(a_{31}+a_{32})(y-1)=0,
\end{equation*}
which becomes
\begin{equation*}
 b_2 a_{21} \frac{1-y}{y}+b_3(a_{31}+a_{32})(y-1)=0,
\end{equation*}
owing to \eqref{eq:limsys3}.
Division by $1-y\ne 0$ results in
\begin{equation*}
 b_2 a_{21} \frac{1}{y}-b_3(a_{31}+a_{32})=0,
\end{equation*}
and thus, we have
\begin{equation}\label{eq:lem1aux1}
 y = \frac{b_2 a_{21}}{b_3(a_{31}+a_{32})}.
\end{equation}
Next, we compute the intersection of the ellipse \eqref{eq:limsys2} and the hyperbola \eqref{eq:limsys3}.
We obtain 
\begin{equation*}
 b_2 a_{21}^2 (x^2-1)+b_3(a_{31}+a_{32})^2(y^2-1)=0,
\end{equation*}
by subtracting \eqref{eq:RKorder5} from \eqref{eq:limsys2}, and
utilization of \eqref{eq:limsys3} yields
\begin{equation*}
 b_2 a_{21}^2 \frac{1-y^2}{y^2}+b_3(a_{31}+a_{32})^2(y^2-1)=0.
\end{equation*}
Owing to $1-y^2\ne0$, as $0<y\ne 1$, we can divide by $1-y^2$ and find
\begin{equation}\label{eq:lem1aux2}
 y=\frac{\sqrt{b_2} a_{21}}{\sqrt{b_3}(a_{31}+a_{32})}.
\end{equation}
Altogether, owing to \eqref{eq:lem1aux1} and \eqref{eq:lem1aux2}, only $b_2=b_3$ yields a potential second solution of \eqref{eq:limsys}.
This solution reads
\begin{equation}\label{eq:sol2}
 x = \frac{a_{31}+a_{32}}{a_{21}},\quad y=\frac{a_{21}}{a_{31}+a_{32}}.
\end{equation}
The remaining question is, if there are any explicit third order Runge-Kutta schemes with non-negative parameters that satisfy $b_2=b_3$.
According to Lemma~\ref{lem:RK3pos}, we have to consider two cases to answer this question. In case I, $b_2=b_3$ can be written as
\[\frac{3\beta-2}{6\alpha(\beta-\alpha)}=\frac{2-3\alpha}{6\beta(\beta-\alpha)}.\]
This is satisfied if 
\[
3\beta^2-2\beta+3\alpha^2-2\alpha=0,
\]
which can be reformulated as 
\begin{equation}\label{eq:circle}
\left(\alpha-\frac13\right)^2+\left(\beta-\frac13\right)^2=\frac29
\end{equation}
holds true.
Thus, $(\alpha,\beta)$ must be a point on the boundary of the circle with center $(1/3,1/3)$ and radius $\sqrt2/3$. 
 \begin{figure}
  \centering
  \includegraphics[width=.8\textwidth]{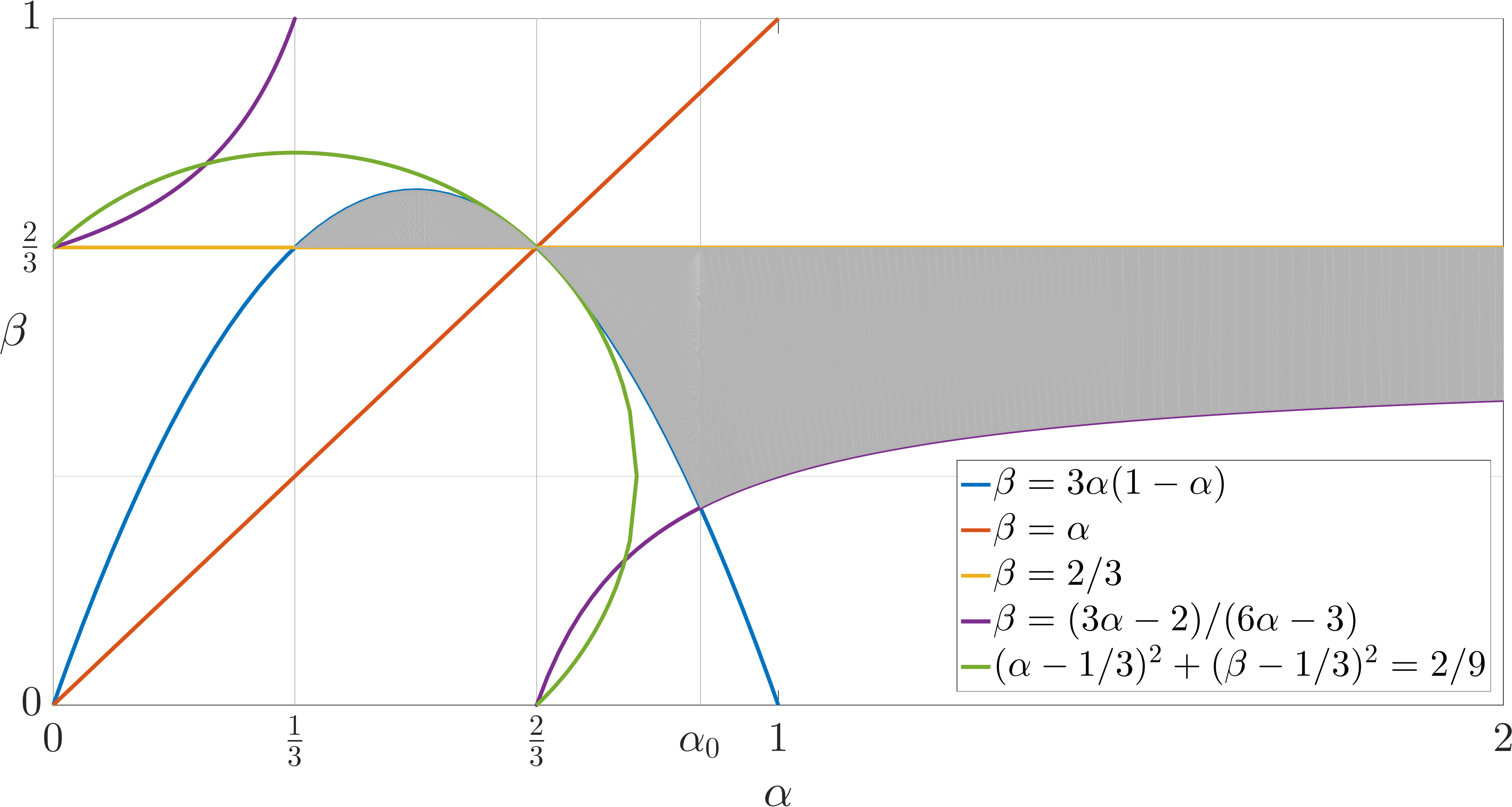}
  \caption{Feasible region (gray) that contains all pairs $(\alpha,\beta)$, for which the Runge-Kutta parameters in case I of Lemma~\ref{lem:RK3} are non-negative and the circle $(\alpha-1/3)^2+(\beta-1/3)^2=2/9$, whose boundary points satisfy $b_2=b_3.$}
  \label{fig:RK3pos2}
 \end{figure}
Figure~\ref{fig:RK3pos2} shows the feasible region from Lemma~\ref{lem:RK3pos}, together with the circle \eqref{eq:circle}.
Computing the intersection of the circle \eqref{eq:circle} and the parabola $3\alpha(1-\alpha)$, yields $\alpha=2/3$. 
As this value of $\alpha$ is excluded in case I, there is no solution of the system \eqref{eq:limsys} in the situation of case I.

In case II of Lemma~\ref{lem:RK3pos}, $b_2=b_3$ is equivalent to $3/4-\gamma=\gamma$, which is satisfied for $\gamma=3/8$. 
Due to $a_{21}=a_{31}+a_{32}=2/3$, \eqref{eq:sol2} becomes \eqref{eq:sol1}. 
All things considered, we have shown that \eqref{eq:sol1} is the unique positive solution of \eqref{eq:limsys}.
\end{proof}
An MPRK scheme \eqref{eq:MPRK} with three stages is given by
\begin{subequations}\label{eq:MPRK33}
\begin{align}
  &\begin{aligned}
    \mathllap{y_i^{(1)}} &= y_i^n,
  \end{aligned}\\ \
  &\begin{aligned} \label{eq:MPRK33c2}
    \mathllap{y_i^{(2)}} &= y_i^n + a_{21}\Delta t\sum_{j=1}^N\biggl( p_{ij}(\mbfy^{(1)})(1-\delta)+p_{ij}(\mbfy^{(1)})\frac{y_j^{(2)}}{\pi_j}\delta-d_{ij}(\mbfy^{(1)})\frac{y_i^{(2)}}{\pi_i}\biggr),
  \end{aligned}\\ 
  &\begin{multlined}[b][.7\columnwidth] \label{eq:MPRK33c3}
    \mathllap{y_i^{(3)}} = y_i^n + \Delta t\sum_{j=1}^N\biggl(\left(a_{31} p_{ij}(\mbfy^{(1)})+a_{32}p_{ij}(\mbfy^{(2)})\right)(1-\delta)\\
    +\left(a_{31} p_{ij}(\mbfy^{(1)})+a_{32} p_{ij}(\mbfy^{(2)})\right)\delta\frac{y_j^{(3)}}{\rho_j}\\
         - \left(a_{31} d_{ij}(\mbfy^{(1)})+ a_{32} d_{ij}(\mbfy^{(2)})\right)\frac{y_i^{(3)}}{\rho_i}\biggr),
  \end{multlined}  \\
  &\begin{multlined}[b][.7\columnwidth] \label{eq:MPRK33cnp1}
    \mathllap{y_i^{n+1}} = y_i^n + \Delta t\sum_{j=1}^N\biggl( \left(b_1 p_{ij}(\mbfy^{(1)})+b_2 p_{ij}(\mbfy^{(2)})+b_3 p_{ij}(\mbfy^{(3)})\right)\frac{y_j^{n+1}}{\sigma_j}\\
         - \left(b_1 d_{ij}(\mbfy^{(1)})+ b_2 d_{ij}(\mbfy^{(2)})+ b_3 d_{ij}(\mbfy^{(3)})\right)\frac{y_i^{n+1}}{\sigma_i}\biggr),
  \end{multlined}  
\end{align}
\end{subequations}
for $i=1,\dots,N$.
The next theorem gives necessary and sufficient conditions for the Patankar-weights of a third order three stage MPRK scheme.
\begin{thm}\label{thm:MPRK33order}
Given an explicit three-stage third order Runge-Kutta scheme with non-negative weights,
 the MPRK scheme \eqref{eq:MPRK33} is of third order, if and only if the conditions
 \begin{subequations}\label{eq:MPRK33order}
 \begin{align}\label{eq:MPRK33order1}
  \pi_i&=y_i^n+\O(\Delta t),\quad i=1,\dots,N,\\ \label{eq:MPRK33order2}
  \rho_i&=y_i^n+\O(\Delta t),\quad i=1,\dots,N,
  \end{align} 
\begin{multline}
 \label{eq:MPRK33order3}
 b_2 a_{21}\frac{y_i^n+a_{21}\Delta t(P_i^n-D_i^n)}{\pi_i}+b_3(a_{31}+a_{32})\frac{y_i^n+(a_{31}+a_{32})\Delta t(P_i^n-D_i^n)}{\rho_i}\\=\frac12+\O(\Delta t^2),\quad i=1,\dots,N,
 \end{multline}
 \begin{align}
 \SwapAboveDisplaySkip
 \label{eq:MPRK33order4}
 \sigma_i =y_i^n+\Delta t(P_i^n-D_i^n) + \frac{\Delta t^2}{2}\frac{\partial(P_i^n-D_i^n)}{\partial\mbfy }(\mbfP^n-\mbfD^n)+\O(\Delta t^3),\quad i=1,\dots,N,
\end{align}
\end{subequations}
are satisfied.
\end{thm}

\begin{proof}
We use the notation $\phi^*$ to represent $\phi(\mbfy^*)$ for a given function $\phi$.
As \eqref{eq:MPRK33} is an MPRK scheme, all Patankar-weights are positive, i.\,e. $\pi_i>0$, $\rho_i>0$ and $\sigma_i>0$ for $i=1,\dots,N$. 

The Runge-Kutta scheme is of third order and substitution of  \eqref{eq:RKorder1} and \eqref{eq:RKorder2} into \eqref{eq:RKorder4} and \eqref{eq:RKorder5} shows
 \begin{subequations}\label{eq:RK3cond00}
 \begin{align}\label{eq:RK3condA}
 b_2 a_{21}+b_3(a_{31}+a_{32})&=\frac12,\\
 b_2 a_{21}^2+b_3(a_{31}+a_{32})^2&=\frac13\label{eq:RK3condB}
 \end{align}
 \end{subequations}
hold true.
Furthermore, equation \eqref{eq:RKorder6} ensures $a_{21},a_{31},b_3>0$.

For a sufficiently smooth function $\phi$ and some $y_i^*=y_i^n+\O(\Delta t)$, we can expand $\phi(\mbfy^*)$ in the form
\begin{equation}\label{eq:MPRK33aux01}
\phi(\mbfy^*)=\phi^n+\frac{\partial\phi^n}{\partial\mbfy}(\mbfy^{*}-\mbfy^n)+\frac12(\mbfy^{*}-\mbfy^n)^T\mbfH_\phi^n(\mbfy^{*}-\mbfy^n)+\O(\Delta t^3),
\end{equation}
in which $\mbfH_{\phi}^n$ denotes the Hessian matrix of $\phi$ evaluated at $\mbfy^n$.
The Taylor expansion of the exact solution of \eqref{eq:pds} reads
\begin{multline}\label{eq:MPRK33exact}
y_i(t^{n+1})=y_i^n+\Delta t(P_i^n-D_i^n)+\frac{\Delta t^2}2\frac{\partial(P_i^n-D_i^n)}{\partial\mbfy}(\mbfP^n-\mbfD^n)\\+ \frac{\Delta t^3}{6}\sum_{k=1}^N\frac{\partial(P_i^n-D_i^n)}{\partial y_k}\frac{\partial(P_k^n-D_k^n)}{\partial \mbfy}(\mbfP^n-\mbfD^n)\\
+\frac{\Delta t^3}{6}(\mbfP^n-\mbfD^n)^T\mbfH_{P_i-D_i}^n(\mbfP^n-\mbfD^n)+\O(\Delta t^4)
\end{multline}
for $i=1,\dots,N$.

To derive necessary conditions, which allow for third order accuracy, we assume that the MPRK scheme \eqref{eq:MPRK33} is third order accurate, this is
\begin{multline}\label{eq:MPRK33numexp}
y_i^{n+1}=y_i^n+\Delta t(P_i^n-D_i^n)+\frac{\Delta t^2}2\frac{\partial(P_i^n-D_i^n)}{\partial\mbfy}(\mbfP^n-\mbfD^n)\\+ \frac{\Delta t^3}{6}\sum_{k=1}^N\frac{\partial(P_i^n-D_i^n)}{\partial y_k}\frac{\partial(P_k^n-D_k^n)}{\partial \mbfy}(\mbfP^n-\mbfD^n)\\
+\frac{\Delta t^3}{6}(\mbfP^n-\mbfD^n)^T\mbfH_{P_i-D_i}^n(\mbfP^n-\mbfD^n)+\O(\Delta t^4)
\end{multline}
for $i=1,\dots,N$. 
Subtracting \eqref{eq:MPRK33exact} from \eqref{eq:MPRK33numexp} shows
$y_i^{n+1}-y(t_i^{n+1})=\O(\Delta t^4)$.
Utilizing \eqref{eq:MPRK33cnp1} and \eqref{eq:MPRK33exact} this can be written as
\begin{multline}\label{eq:MPRK33aux1}
\sum_{j=1}^N\biggl( \left(b_1 p_{ij}^n+b_2 p_{ij}^{(2)}+b_3 p_{ij}^{(3)}\right)\frac{y_j^{n+1}}{\sigma_j}
         - \left(b_1 d_{ij}^n+ b_2 d_{ij}^{(2)}+ b_3 d_{ij}^{(3)}\right)\frac{y_i^{n+1}}{\sigma_i}\biggr)
-(P_i^n-D_i^n)\\-\frac{\Delta t}2\frac{\partial(P_i^n-D_i^n)}{\partial\mbfy}(\mbfP^n-\mbfD^n)- \frac{\Delta t^2}{6}\sum_{k=1}^N\frac{\partial(P_i^n-D_i^n)}{\partial y_k}\frac{\partial(P_k^n-D_k^n)}{\partial \mbfy}(\mbfP^n-\mbfD^n)\\
-\frac{\Delta t^2}{6}(\mbfP^n-\mbfD^n)^T\mbfH_{P_i-D_i}^n(\mbfP^n-\mbfD^n)
=\O(\Delta t^3)
\end{multline}
for $i=1,\dots,N$.

From now on, we focus on the solution of the PDS \eqref{eq:PDSorder} with
$I,J\in\{1,\dots,N\}$, $I\ne J$, $\kappa\in\{1,2\}$ and $\mu>0$.
Since $\PPDS_I=0$ and $\DPDS_I=\dPDS_{IJ}=\mu y_I^\kappa$, it follows that $\partial D_I/\partial\mbfy=(\partial D_I/\partial y_I) \mathbf e_I^T=\mu \kappa y_I^{\kappa-1}\mathbf e_I^T$, with $\mathbf e_I$ denoting the $I$th unit column vector and $\mbfH_{D_I}=(\partial^2\DPDS_{I}^n/{\partial y_I^2})\mathbf e_I\mathbf e_I^T=\mu \kappa(\kappa-1) y_I^{\kappa-2}\mathbf e_I\mathbf e_I^T$.
Hence, \eqref{eq:MPRK33aux1} with $i=I$ becomes
\begin{multline}\label{eq:MPRK33aux15}
- \left(b_1 \DPDS_{I}^n+ b_2 \DPDS_{I}^{(2)}+ b_3 \DPDS_{I}^{(3)}\right)\frac{y_I^{n+1}}{\sigma_I}
+\DPDS_I^n-\frac{\Delta t}2\frac{\partial\DPDS_I^n}{y_I}\DPDS_I^n\\
+ \frac{\Delta t^2}{6}\frac{\partial\DPDS_I^n}{\partial y_I}\frac{\partial\DPDS_I^n}{\partial y_I}\DPDS_I^n
+\frac{\Delta t^2}{6}\frac{\partial^2\DPDS_{I}^n}{\partial y_I^2}(\DPDS_I^n)^2
=\O(\Delta t^3).
\end{multline}
For $k=1,2$ the destruction terms can be expanded as
\begin{align}\label{eq:MPRK33aux4}
  \DPDS_{I}^{(k)}&=\DPDS_{I}^n+\frac{\partial \DPDS_{I}^n}{\mbfy}(\mbfy^{(k)}-\mbfy^n)+\frac12 (\mbfy^{(k)}-\mbfy^n)^T\mbfH_{\DPDS_{I}}^n(\mbfy^{(k)}-\mbfy^n)\nonumber\\
  &=\DPDS_{I}^n+\frac{\partial \DPDS_{I}^n}{y_I}(y_I^{(k)}-y_I^n)+\frac12 \frac{\partial^2\DPDS_{I}^n}{\partial y_I^2}(y_I^{(k)}-y_I^n)^2,
\end{align}
since derivatives of order higher than two vanish. 
Substituting this into \eqref{eq:MPRK33aux15}, results in
\begin{multline}\label{eq:MPRK33aux16}
-\DPDS_I^n\biggl( (\overbrace{b_1 +b_2+b_3}^{=1})\frac{y_I^{n+1}}{\sigma_I}
-1\biggr)-\frac{\partial\DPDS_I^n}{y_I}\biggl(\left(b_2(y_I^{(2)}-y_I^n)+b_3(y_I^{(3)}-y_I^n)\right)\frac{y_I^{n+1}}{\sigma_I}+\frac{\Delta t}2\DPDS_I^n\biggr)\\
+ \frac{\Delta t^2}{6}\frac{\partial\DPDS_I^n}{\partial y_I}\frac{\partial\DPDS_I^n}{\partial y_I}\DPDS_I^n
-\frac12\frac{\partial^2\DPDS_{I}^n}{\partial y_I^2}\biggl(\left(b_2(y_I^{(2)}-y_I^n)^2+b_3(y_I^{(3)}-y_I^n)^2\right)\frac{y_I^{n+1}}{\sigma_I}-\frac{\Delta t^2}{3}(\DPDS_I^n)^2\biggr)
\\=\O(\Delta t^3).
\end{multline}
Owing to \eqref{eq:MPRK33c2}, we have
\begin{equation}\label{eq:MPRK33aux3}
 y_I^{(2)}-y_I^n = -a_{21}\Delta t \DPDS_I^n \frac{y_I^{(2)}}{\pi_I},
\end{equation}
and
from \eqref{eq:MPRK33c3}, \eqref{eq:MPRK33aux4} and \eqref{eq:MPRK33aux3} we see
\begin{align}\label{eq:MPRK33aux5}
 y_I^{(3)}-y_I^n &= -\Delta t\left(a_{31} \DPDS_I^n+a_{32}\DPDS_I^{(2)}\right)\frac{y_I^{(3)}}{\rho_I}\nonumber\\
 &=-\Delta t\biggl((a_{31}+a_{32})\DPDS_I^n-a_{32}\frac{\partial \DPDS_I^n}{y_I}a_{21}\Delta t \DPDS_I^n\frac{y_I^{(2)}}{\pi_I}+\frac12 \frac{\partial^2\DPDS_{I}^n}{\partial y_I^2}\biggl(a_{21}\Delta t \DPDS_I^n \frac{y_I^{(2)}}{\pi_I}\biggr)^{\!\!2}\biggr)\frac{y_I^{(3)}}{\rho_I}.
\end{align}

Before we introduce \eqref{eq:MPRK33aux3} and \eqref{eq:MPRK33aux5} into \eqref{eq:MPRK33aux16}, we set $\kappa=1$, which implies $(\partial^2\DPDS_{I}^n/{\partial y_I^2})=\mu \kappa(\kappa-1) y_I^{\kappa-2}=0$.
Hence, we can drop the terms containing second derivatives in \eqref{eq:MPRK33aux16} and \eqref{eq:MPRK33aux5}.
We can exploit these conditions, when  we consider \eqref{eq:MPRK33aux16} with $\kappa=2$, as some terms can be neglected.
Setting $\kappa=1$, we have 
\begin{multline*}
 b_2(y_I^{(2)}-y_I^n)+b_3(y_I^{(3)}-y_I^n)=\\-\Delta t \DPDS_I^n \biggl(b_2 a_{21} \frac{y_I^{(2)}}{\pi_I}+b_3(a_{31}+a_{32})\frac{y_I^{(3)}}{\rho_I}-\underbrace{b_3 a_{32}a_{21}}_{=1/6}\Delta t\frac{\partial \DPDS_I^n}{y_I}\frac{y_I^{(2)}}{\pi_I}\frac{y_I^{(3)}}{\rho_I}\biggr)
\end{multline*}
according to \eqref{eq:MPRK33aux3} and \eqref{eq:MPRK33aux5}, since $\partial^2\DPDS_{I}^n/{\partial y_I^2}=0$. 
Substituting this into \eqref{eq:MPRK33aux16} yields
\begin{multline*}
-\DPDS_I^n\biggl(\frac{y_I^{n+1}}{\sigma_I}
-1\biggr)+\Delta t\DPDS_I^n\frac{\partial\DPDS_I^n}{y_I}\biggl(\biggl(b_2 a_{21} \frac{y_I^{(2)}}{\pi_I}+b_3(a_{31}+a_{32})\frac{y_I^{(3)}}{\rho_I}\biggr)\frac{y_I^{n+1}}{\sigma_I}-\frac12\biggr)\\
+ \frac{\Delta t^2}{6}\frac{\partial\DPDS_I^n}{\partial y_I}\frac{\partial\DPDS_I^n}{\partial y_I}\DPDS_I^n\biggl(1-\frac{y_I^{(2)}}{\pi_I}\frac{y_I^{(3)}}{\rho_I}\frac{y_I^{n+1}}{\sigma_I}\biggr)=\O(\Delta t^3).
\end{multline*}
A subsequent division by $-\DPDS_I^n=-\mu y_I^n\ne 0$ and utilization of $\partial \DPDS_I/\partial y_I=\mu \kappa y_I^{\kappa-1}=\mu$ results in
\begin{multline*}
\frac{y_I^{n+1}}{\sigma_I}
-1-\Delta t\mu\biggl(\biggl(b_2 a_{21} \frac{y_I^{(2)}}{\pi_I}+b_3(a_{31}+a_{32})\frac{y_I^{(3)}}{\rho_I}\biggr)\frac{y_I^{n+1}}{\sigma_I}-\frac12\biggr)\\
- \frac{\Delta t^2}{6}\mu^2\biggl(1-\frac{y_I^{(2)}}{\pi_I}\frac{y_I^{(3)}}{\rho_I}\frac{y_I^{n+1}}{\sigma_I}\biggr)=\O(\Delta t^3).
\end{multline*}
Since $\mu>0$ was chosen arbitrary, we find that this holds true for all $\mu>0$.
From Lemma~\ref{lem:lemO} we can conclude that
\begin{align}\label{eq:MPRK33aux10}
 \frac{y_I^{n+1}}{\sigma_I}&=1+\O(\Delta t^3),\\ \label{eq:MPRK33aux11}
 \biggl( b_2 a_{21}\frac{y_I^{(2)}}{\pi_I}+b_3 (a_{31}+a_{32})\frac{y_I^{(3)}}{\rho_I}\biggr) \frac{y_I^{n+1}}{\sigma_I}&=\frac{1}2+\O(\Delta t^2),\\ \label{eq:MPRK33aux12}
 \frac{y_I^{(2)}}{\pi_I}\frac{y_I^{(3)}}{\rho_I}\frac{y_I^{n+1}}{\sigma_I}    
         &= 1+\O(\Delta t)
\end{align}
hold true.

The above equations contain products of Patankar-weights.
To find conditions for the PWDs, we determine the limits of the Patankar-weights.
In this regard, equation \eqref{eq:MPRK33aux10} shows $y_I^{n+1}/\sigma_I\to 1$.
Substitution of this into \eqref{eq:MPRK33aux11} and \eqref{eq:MPRK33aux12} yields 
\begin{equation}\label{eq:MPRK33aux18a}
b_2 a_{21}\frac{y_I^{(2)}}{\pi_I}+b_3 (a_{31}+a_{32})\frac{y_I^{(3)}}{\rho_I}\to\frac12
\end{equation}
and 
\begin{equation}\label{eq:MPRK33aux18b}
 \frac{y_I^{(2)}}{\pi_I}\frac{y_I^{(3)}}{\rho_I}\to 1.
\end{equation}
Next, we show that none of the Patankar-weights ${y_I^{(2)}}/{\pi_I}$ and ${y_I^{(3)}}/{\rho_I}$ can tend to infinity.
To do so, we must consider two cases. 
If $b_2=0$, \eqref{eq:RK3condA} becomes $b_3(a_{31}+a_{32})=1/2$, so we can conclude $y_I^{(3)}/\rho_I\to 1$ from \eqref{eq:MPRK33aux18a} and thus, $y_I^{(2)}/\pi_I\to 1$ from \eqref{eq:MPRK33aux18b}.
If $b_2> 0$, both terms on the left hand side of \eqref{eq:MPRK33aux18a} are positive, since $a_{21},a_{32},b_3>0$ and $a_{31}\geq 0$, hence, $y_I^{(2)}/\pi_I\not\to\infty$ and ${y_I^{(3)}}/{\rho_I}\not\to\infty$. 
Consequently, owing to \eqref{eq:MPRK33aux18b}, we find that none of the Patankar-weights ${y_I^{(2)}}/{\pi_I}$ or ${y_I^{(3)}}/{\rho_I}$ can tend to zero, as this would require the other weight to tend to infinity.
Denoting by $\Gamma_I^{(2)}$ and $\Gamma_I^{(3)}$ the limits of $y_I^{(2)}/\pi_I$ and $y_I^{(3)}/\rho_I$, we have
\begin{equation}\label{eq:MPRK33aux17}
\frac{y_I^{(2)}}{\pi_I}\to \Gamma_I^{(2)},\quad \frac{y_I^{(3)}}{\rho_I}\to \Gamma_I^{(3)},
\end{equation}
with $\Gamma_I^{(2)},\Gamma_I^{(3)}>0$ and 
\begin{align}\label{eq:MPRK33aux19a}
\Gamma_I^{(2)} \Gamma_I^{(3)}&=1,\\\label{eq:MPRK33aux19a2}
b_2 a_{21} \Gamma_I^{(2)}+b_3(a_{31}+a_{32})\Gamma_I^{(3)}&=1/2. 
\end{align}

Now we consider the case $\kappa=2$ in  \eqref{eq:MPRK33aux16}.
From \eqref{eq:MPRK33aux5} and \eqref{eq:MPRK33aux17} we see
\begin{align*}
 y_I^{(3)}-y_I^n &=-\Delta t\biggl((a_{31}+a_{32})\DPDS_I^n-a_{32}\frac{\partial \DPDS_I^n}{y_I}a_{21}\Delta t \DPDS_I^n\frac{y_I^{(2)}}{\pi_I}+\overbrace{\frac12 \frac{\partial^2\DPDS_{I}^n}{\partial y_I^2}\biggl(a_{21}\Delta t \DPDS_I^n \underbrace{\frac{y_I^{(2)}}{\pi_I}}_{\mathclap{\eq\limits_{\eqref{eq:MPRK33aux17}}\O(1)}}\biggr)^{\!\!2}}^{=\O(\Delta t^2)}\biggr)\frac{y_I^{(3)}}{\rho_I},\\
 &=-\Delta t\biggl((a_{31}+a_{32})\DPDS_I^n-\underbrace{a_{32}\frac{\partial \DPDS_I^n}{y_I}a_{21}\Delta t \DPDS_I^n\frac{y_I^{(2)}}{\pi_I}}_{=\O(\Delta t)}\biggr)\frac{y_I^{(3)}}{\rho_I}+\O(\Delta t^3),
\end{align*}
which implies
\begin{align*}
 (y_I^{(3)}-y_I^n)^2 
 &=\Delta t^2 (a_{31}+a_{32})^2(\DPDS_I^n)^2\biggl(\frac{y_I^{(3)}}{\rho_I}\biggr)^{\!\!2}+\O(\Delta t^3).
\end{align*}
Together with \eqref{eq:MPRK33aux3} we find
\begin{multline*}
 b_2(y_I^{(2)}-y_I^n)+b_3(y_I^{(3)}-y_I^n)=\\-\Delta t \DPDS_I^n \biggl(b_2 a_{21} \frac{y_I^{(2)}}{\pi_I}+b_3(a_{31}+a_{32})\frac{y_I^{(3)}}{\rho_I}-\frac{\Delta t}{6}\frac{\partial \DPDS_I^n}{y_I}\frac{y_I^{(2)}}{\pi_I}\frac{y_I^{(3)}}{\rho_I}\biggr)+\O(\Delta t^3)
\end{multline*}
and
\begin{multline*}
 b_2(y_I^{(2)}-y_I^n)^2+b_3(y_I^{(3)}-y_I^n)^2=\\
 \Delta t^2 (\DPDS_I^n)^2\biggl(b_2 a_{21}^2 \biggl(\frac{y_I^{(2)}}{\pi_I}\biggr)^{\!\!2}+b_3(a_{31}+a_{32})^2\biggl(\frac{y_I^{(3)}}{\rho_I}\biggr)^{\!\!2}\biggr)+\O(\Delta t^3).
\end{multline*}
Substituting this and \eqref{eq:MPRK33aux10} into \eqref{eq:MPRK33aux16} 
yields
\begin{multline*}
\Delta t\DPDS_I^n\frac{\partial\DPDS_I^n}{y_I}\biggl(\left(b_2 a_{21} \frac{y_I^{(2)}}{\pi_I}+b_3(a_{31}+a_{32})\frac{y_I^{(3)}}{\rho_I}\right)\frac{y_I^{n+1}}{\sigma_I}-\frac12\biggr)\\
+ \frac{\Delta t^2}{6}\frac{\partial\DPDS_I^n}{\partial y_I}\frac{\partial\DPDS_I^n}{\partial y_I}\DPDS_I^n\biggl(1-\frac{y_I^{(2)}}{\pi_I}\frac{y_I^{(3)}}{\rho_I}\frac{y_I^{n+1}}{\sigma_I}\biggr)\\
-\frac12\Delta t^2 (\DPDS_I^n)^2\frac{\partial^2\DPDS_{I}^n}{\partial y_I^2}\biggl(\biggl(b_2 a_{21}^2 \biggl(\frac{y_I^{(2)}}{\pi_I}\biggr)^{\!\!2}+b_3(a_{31}+a_{32})^2\biggl(\frac{y_I^{(3)}}{\rho_I}\biggr)^{\!\!2}\biggr)\frac{y_I^{n+1}}{\sigma_I}-\frac13\biggr)\\
=\O(\Delta t^3).
\end{multline*}
Owing to \eqref{eq:MPRK33aux11} and \eqref{eq:MPRK33aux12}, we find
\begin{equation*}
-\frac12\Delta t^2 (\DPDS_I^n)^2\frac{\partial^2\DPDS_{I}^n}{\partial y_I^2}\biggl(\biggl(b_2 a_{21}^2 \biggl(\frac{y_I^{(2)}}{\pi_I}\biggr)^{\!\!2}+b_3(a_{31}+a_{32})^2\biggl(\frac{y_I^{(3)}}{\rho_I}\biggr)^{\!\!2}\biggr)\frac{y_I^{n+1}}{\sigma_I}-\frac13\biggr)
=\O(\Delta t^3),
\end{equation*}
which shows
\begin{equation*}
b_2 a_{21}^2 \biggl(\frac{y_I^{(2)}}{\pi_I}\biggr)^{\!\!2}+b_3(a_{31}+a_{32})^2\biggl(\frac{y_I^{(3)}}{\rho_I}\biggr)^{\!\!2}
=\frac13+\O(\Delta t),
\end{equation*}
due to \eqref{eq:MPRK33aux10} and \eqref{eq:MPRK33aux17}.
Utilization of \eqref{eq:MPRK33aux17} results in an additional equation 
\begin{equation}\label{eq:MPRK33aux19b}
b_2 a_{21}^2 (\Gamma_I^{(2)})^2+b_3(a_{31}+a_{32})^2 (\Gamma_I^{(3)})^2=\frac13. 
\end{equation}
containing $\Gamma_I^{(2)}$ and $\Gamma_I^{(3)}$.
To determine the values $\Gamma_I^{(2)}$ and $\Gamma_I^{(3)}$ we can use Lemma~\ref{lem:MPRK33sys}, which shows that the system \eqref{eq:MPRK33aux19a}, \eqref{eq:MPRK33aux19a2} and \eqref{eq:MPRK33aux19b} has the unique solution
\[
\Gamma_I^{(2)}=\Gamma_I^{(3)}=1.
\]
Together with \eqref{eq:MPRK33aux10} and \eqref{eq:MPRK33aux12}, this leads to 
\begin{align}\label{eq:MPRK33aux23}
 \frac{y_I^{(2)}}{\pi_I}=1+\O(\Delta t),\quad
 \frac{y_I^{(3)}}{\rho_I}=1+\O(\Delta t).
\end{align}
Substituting this into \eqref{eq:MPRK33aux3} and \eqref{eq:MPRK33aux5}, we find $y_I^{(2)}=y_I^n+\O(\Delta t)$ and $y_I^{(3)}=y_I^n+\O(\Delta t)$.
Hence, we obtain
\begin{equation}\label{eq:MPRKIIIaux20}
 1+\O(\Delta t)=\frac{1}{1+\O(\Delta t)}=\frac{\pi_I}{y_I^{(2)}}=\frac{\pi_I}{y_I^n+\O(\Delta t)}=\frac{\pi_I}{y_I^n}+\O(\Delta t),
\end{equation}
from which we conclude
\begin{equation}\label{eq:MPRKIIIaux19a}
 \pi_I=y_I^n+\O(\Delta t).
\end{equation}
In an analogous manner, we can conclude
\begin{equation}\label{eq:MPRKIIIaux19b}
\rho_I=y_I^n+\O(\Delta t).
\end{equation}
from \eqref{eq:MPRK33aux23}.
Since $I$ was chosen arbitrary, we can let it run from 1 to $N$ and find that \eqref{eq:MPRK33order1} and \eqref{eq:MPRK33order2} are necessary conditions.

The same is true for the other equations derived above. 
In particular,
\begin{align}\label{eq:MPRK33aux20a}
 \frac{y_i^{(2)}}{\pi_i}&=1+\O(\Delta t),\quad i=1,\dots,N,\\
 \frac{y_i^{(3)}}{\rho_i}&=1+\O(\Delta t),\quad i=1,\dots,N,\label{eq:MPRK33aux20b}
\end{align}
hold true.
These equations are helpful to find a concise representation of \eqref{eq:MPRK33aux11}. 
On account of \eqref{eq:MPRK33c2} and \eqref{eq:MPRK33aux20a}, we have
\begin{equation}\label{eq:MPRK33aux21}
 y_i^{(2)}=y_i^n + a_{21}\Delta t(\PPDS_i^n-\DPDS_i^n)+\O(\Delta t^2)
\end{equation}
for $i=1,\dots,N$.
Similar, \eqref{eq:MPRK33c3} and \eqref{eq:MPRK33aux20b} show
\begin{equation}\label{eq:MPRK33aux22}
y_i^{(3)}=y_i^n + \Delta t\bigl(a_{31}(\PPDS_i^n-\DPDS_i^n)+a_{32}(\PPDS_i^{(2)}-\DPDS_i^{(2)})\bigr) +\O(\Delta t^2),
\end{equation}
for $i=1,\dots,N$.
According to \eqref{eq:MPRK33aux21}, we have $\mbfy^{(2)}-\mbfy^n=\O(\Delta t)$, and from \eqref{eq:MPRK33aux01} we see
\begin{equation*}
 \PPDS_i^{(2)}-\DPDS_i^{(2)}=\PPDS_i^n-\DPDS_i^n+\O(\Delta t).
\end{equation*}
Hence, 
\begin{equation*}
y_i^{(3)}=y_i^n + \Delta t(a_{31}+a_{32})(\PPDS_i^n-\DPDS_i^n) +\O(\Delta t^2),
\end{equation*}
follows from \eqref{eq:MPRK33aux22}.
Substituting this into \eqref{eq:MPRK33aux11}, and taking account of \eqref{eq:MPRKIIIaux19a} and \eqref{eq:MPRKIIIaux19b}, results in
\begin{equation*}
 b_2 a_{21}\frac{y_I^n+a_{21}\Delta t(\PPDS_I^n-\DPDS_I^n)}{\pi_I}+b_3(a_{31}+a_{32})\frac{y_I^n+(a_{31}+a_{32})\Delta t(\PPDS_I^n-\DPDS_I^n)}{\rho_I}=\frac12+\O(\Delta t^2).
\end{equation*}
Again, $I$ was chosen arbitrary, and letting it run from $1$ to $N$, shows that condition \eqref{eq:MPRK33order3} is necessary.
Finally, analogous to \eqref{eq:MPRKIIIaux20}, \eqref{eq:MPRK33aux10} and \eqref{eq:MPRK33numexp} show
\begin{equation*}
 \sigma_I=y_I^n+\Delta t(\PPDS_I^n-\DPDS_I^n) + \frac{\Delta t^2}{2}\frac{\partial (\PPDS_I^n-\DPDS_I^n)}{\partial \mbfy}(\mathbf \PPDS^n-\mathbf \DPDS^n)+\O(\Delta t^3).
\end{equation*}
By letting $I$ run from $1$ to $N$, we see that also condition \eqref{eq:MPRK33order4} is necessary.

Now we show that the conditions \eqref{eq:MPRK33order} are sufficient, to make \eqref{eq:MPRK33} a third order MPRK scheme.
We start our investigation with the choice $\delta = 1$. 
The MPRK scheme \eqref{eq:MPRK33} can be written in the form of three linear systems
\[
\mbfM^{(2)}\mbfy^{(2)}=\mbfy^n,\quad \mbfM^{(3)}\mbfy^{(3)}=\mbfy^n,\quad\mbfM\mbfy^{n+1}=\mbfy^{n}.
\]
Since $\delta=1$, utilizing Lemma~\ref{lem:mbound} yields $(\mbfM^{(2)})^{-1}=\O(1)$, $(\mbfM^{(3)})^{-1}=\O(1)$ and $\mbfM^{-1}=\O(1)$.
Thus, we can conclude $\mbfy^{(2)}=\O(1)$, $\mbfy^{(3)}=\O(1)$ and $\mbfy^{n+1}=\O(1)$.
Together with conditions \eqref{eq:MPRK33order1}, \eqref{eq:MPRK33order2}, and \eqref{eq:MPRK33order4}, this results in
\begin{align}\label{eq:MPRKIIIaux7}
 \frac{y_i^{(2)}}{\pi_i}&=\O(1),\\
\label{eq:MPRKIIIaux8}
 \frac{y_i^{(3)}}{\rho_i}&=\O(1),\\
 \frac{y_i^{n+1}}{\sigma_i}&=\O(1) \label{eq:MPRKIIIaux8b}
\end{align}
for $i=1,\dots,N$, since $y_i^n>0$.
The boundedness of the Patankar-weights \eqref{eq:MPRKIIIaux7} shows that \eqref{eq:MPRK33c2} yields
\begin{equation}\label{eq:MPRKIIIaux9}
y_i^{(2)}=y_i^n + a_{21}\Delta t\sum_{j=1}^N\biggl( p_{ij}^n\frac{y_j^{(2)}}{\pi_j}-d_{ij}^n\frac{y_i^{(2)}}{\pi_i}\biggr)=y_i^n+\O(\Delta t)
\end{equation}
for $i=1,\dots,N$.
This allows us to use \eqref{eq:MPRK33aux01} to expand $p_{ij}(\mbfy^{(2)})$ and $d_{ij}(\mbfy^{(2)})$ in the form
\begin{equation}\label{eq:MPRKIIIaux18}
p_{ij}(\mbfy^{(2)})=p_{ij}^n+\O(\Delta t)=\O(1),\quad
d_{ij}(\mbfy^{(2)})=d_{ij}^n+\O(\Delta t)=\O(1)
\end{equation}
for $i,j=1,\dots,N$.
Substituting this into \eqref{eq:MPRK33c3}, and taking account of \eqref{eq:MPRKIIIaux8}, we find 
\begin{multline}\label{eq:MPRKIIIaux17}
 y_i^{(3)} = y_i^n + \Delta t\sum_{j=1}^N\biggl((a_{31}p_{ij}^n+a_{32}p_{ij}^{(2)})\frac{y_j^{(3)}}{\rho_j}-(a_{31}d_{ij}^{n}+a_{32}d_{ij}^{(2)})\frac{y_i^{(3)}}{\rho_i}\biggr)
 =y_i^n+\O(\Delta t)
\end{multline}
as well. 
Now, we show that \eqref{eq:MPRKIIIaux9} and \eqref{eq:MPRKIIIaux17} are valid for $\delta =0$ as well.
In this case, owing to \eqref{eq:MPRK33c2}, we have
\begin{equation*}
 y_i^{(2)} = \frac{y_i^n+a_{21}\Delta t P_i^n}{1+a_{21}\Delta t D_i^n/\pi_i}
\end{equation*}
for $i=1,\dots,N$.
Since $1/\pi_i=\O(1)$ according to  \eqref{eq:MPRK33order1}, we can conclude $y_i^{(2)}=\O(1)$ and thus $y_i^{(2)}/\pi_i=\O(1)$ for $i=1,\dots,N$.
Utilizing this in \eqref{eq:MPRK33c2} we find
\begin{equation*}
 y_i^{(2)}=y_i^n + a_{21}\Delta t\biggl(P_i^n-D_i^n\frac{y_i^{(2)}}{\pi_i}\biggr)=y_i^n+\O(\Delta t),
\end{equation*}
as before in \eqref{eq:MPRKIIIaux9}. Hence, \eqref{eq:MPRKIIIaux18} is holds true for $\delta=0$ as well.
This, together with \eqref{eq:MPRK33order2}, shows
\begin{equation*}
 y_i^{(3)} = \frac{y_i^n+\Delta t(a_{31}P_i^n+a_{32}P_i^{(2)})}{1+\Delta t(a_{31}D_i^n+a_{32}D_i^{(2)})/\rho_i}=\O(1)
\end{equation*}
and in addition $y_i^{(3)}/\rho_i=\O(1)$ for $i=1,\dots,N$.
Consequently, we have
\begin{equation*}
 y_i^{(3)} = y_i^n + \Delta t\biggl(a_{31}P_i^n+a_{32}P_i^{(2)}-(a_{31}D_i^n+a_{32}D_i^{(2)})\frac{y_i^{(3)}}{\rho_i}\biggr)=y_i^n+\O(\Delta t)
\end{equation*}
for $i=1,\dots,N$, as in \eqref{eq:MPRKIIIaux17}.
The remaining part of the proof is independent of the value of $\delta$.

Owing to \eqref{eq:MPRKIIIaux9} and  \eqref{eq:MPRKIIIaux17}, \eqref{eq:MPRK33aux01} shows
\begin{subequations}\label{eq:MPRK33aux24}
\begin{align}
p_{ij}(\mbfy^{(k)})&=p_{ij}^n+\frac{\partial p_{ij}^n}{\partial\mbfy}(\mbfy^{(k)}-\mbfy^n)+\frac12(\mbfy^{(k)}-\mbfy^n)^T\mbfH_{p_{ij}}^n(\mbfy^{(k)}-\mbfy^n)+\O(\Delta t^3),\\
d_{ij}(\mbfy^{(k)})&=d_{ij}^n+\frac{\partial d_{ij}^n}{\partial\mbfy}(\mbfy^{(k)}-\mbfy^n)+\frac12(\mbfy^{(k)}-\mbfy^n)^T\mbfH_{d_{ij}}^n(\mbfy^{(k)}-\mbfy^n)+\O(\Delta t^3)
\end{align}
\end{subequations}
for $i,j=1,\dots,N$ and $k=2,3$.
Utilizing this and \eqref{eq:MPRKIIIaux8b}, the solution on the next time level \eqref{eq:MPRK33cnp1} satisfies
\begin{equation*}
y_i^{n+1}=y_i^n+\O(\Delta t),
\end{equation*}
for $i=1,\dots,N$. 
Hence, we can conclude
\begin{equation*}
 \frac{y_i^{n+1}}{\sigma_i}=1+\O(\Delta t)
\end{equation*}
from \eqref{eq:MPRK33order4}.
Inserting this and \eqref{eq:MPRK33aux24} into \eqref{eq:MPRK33cnp1} shows
\begin{equation*}
 y_i^{n+1}=y_i^n+\Delta t(P_i^n-D_i^n)+\O(\Delta t^2)
\end{equation*}
for $i=1,\dots,N$, since $b_1+b_2+b_3=1$ according to \eqref{eq:RKorder3}.
Now, we can conclude 
\begin{equation*}
 \frac{y_i^{n+1}}{\sigma_i}=1+\O(\Delta t^2)
\end{equation*}
from \eqref{eq:MPRK33order4}. 
Introducing this relation and \eqref{eq:MPRK33aux24} into \eqref{eq:MPRK33cnp1} yields
\begin{equation}\label{eq:MPRKIIIaux12}
y_i^{n+1}=y_i^n + \Delta t(P_i^n-D_i^n) + \Delta t\frac{\partial(P_i^n-D_i^n)}{\partial\mbfy}\bigl(b_2(\mbfy^{(2)}-\mbfy^n) + b_3(\mbfy^{(3)}-\mbfy^n)\bigr)+\O(\Delta t^3) 
\end{equation}
for $i=1,\dots,N$.
From \eqref{eq:MPRKIIIaux9} and \eqref{eq:MPRK33order1} we can conclude
\begin{equation*}
 \frac{y_i^{(2)}}{\pi_i}=1+\O(\Delta t),
\end{equation*}
thus, \eqref{eq:MPRK33c2} shows
\begin{equation}\label{eq:MPRKIIIaux10}
 y_i^{(2)}=y_i^n+a_{21}\Delta t(P_i^n-D_i^n)+ \O(\Delta t^2).
\end{equation}
for $i=1,\dots,N$.
Similar, \eqref{eq:MPRKIIIaux8} and \eqref{eq:MPRK33c3} imply
\begin{equation}\label{eq:MPRKIIIaux16}
 \frac{y_i^{(3)}}{\rho_i} = 1+\O(\Delta t)
\end{equation}
for $i=1,\dots,N$.
Thus, inserting this and \eqref{eq:MPRK33aux24} into \eqref{eq:MPRK33c3} shows
\begin{equation}\label{eq:MPRKIIIaux11}
 y_i^{(3)}=y_i^n+(a_{31}+a_{32})\Delta t(P_i^n-D_i^n)+\O(\Delta t^2)
\end{equation}
for $i=1,\dots,N$.
Finally, substitution of \eqref{eq:MPRKIIIaux10} and \eqref{eq:MPRKIIIaux11} into \eqref{eq:MPRKIIIaux12} results in
\begin{equation*}
 y_i^{n+1}=y_i^n+\Delta t(P_i^n-D_i^n)+\frac{\Delta t^2}{2}\frac{\partial(P_i^n-D_i^n)}{\partial\mbfy}(\mbfP^n-\mbfD^n)+\O(\Delta t^3),
\end{equation*}
since $b_2a_{21}+b_3(a_{31}+a_{32})=1/2$ due to \eqref{eq:RKorder0}.
Hence, we even have
\begin{equation}\label{eq:MPRKIIIaux13}
 \frac{y_i^{n+1}}{\sigma_i}=1+\O(\Delta t^3)
\end{equation}
for $i=1,\dots,N$.

This enables the proof of the third order accuracy of the MPRK scheme. 
Substitution of \eqref{eq:MPRKIIIaux13} and \eqref{eq:MPRK33aux24} into \eqref{eq:MPRK33cnp1} yields
\begin{multline*}
y_i^{n+1}=y_i^n+\Delta t(P_i^n-D_i^n)+\Delta t\frac{\partial(P_i^n-D_i^n)}{\partial\mbfy}\bigl(b_2(\mbfy^{(2)}-\mbfy^n)+b_3(\mbfy^{(3)}-\mbfy^n)\bigr)\\
+\frac{\Delta t}{2}\bigl(b_2(\mbfy^{(2)}-\mbfy^n)^T\mbfH_{P_i-D_i}^n(\mbfy^{(2)}-\mbfy^n)+b_3(\mbfy^{(3)}-\mbfy^n)^T\mbfH_{P_i-D_i}^n(\mbfy^{(3)}-\mbfy^n)\bigr)+\O(\Delta t^4).
\end{multline*}
Taking account of \eqref{eq:MPRKIIIaux10} and \eqref{eq:MPRKIIIaux11}, and using $b_2 a_{21}^2+b_3(a_{31}+a_{32})^2=1/3$, this can be written in the form 
\begin{multline}\label{eq:MPRKIIIaux15}
y_i^{n+1}=y_i^n+\Delta t(P_i^n-D_i^n)+\Delta t\frac{\partial(P_i^n-D_i^n)}{\partial\mbfy}\bigl(b_2(\mbfy^{(2)}-\mbfy^n)+b_3(\mbfy^{(3)}-\mbfy^n)\bigr)\\
+\frac{\Delta t^3}{6}(\mbfP^n-\mbfD^n)^T\mbfH_{P_i-D_i}^n(\mbfP^n-\mbfD^n)+\O(\Delta t^4).
\end{multline}
It remains to expand $b_2(\mbfy^{(2)}-\mbfy^n)+b_3(\mbfy^{(3)}-\mbfy^n)$ up to $\O(\Delta t^3)$.
Therefore, we use \eqref{eq:MPRK33c3} and \eqref{eq:MPRK33aux24} to see
\begin{multline*}
 y_i^{(3)}=y_i^n + (a_{31}+a_{32})\Delta t\sum_{j=1}^N\biggl(p_{ij}^n\biggl(1-\delta+\delta\frac{y_j^{(3)}}{\rho_j}\biggr)-d_{ij}^n\frac{y_i^{(3)}}{\rho_i}\biggr)\\
 +a_{32}\Delta t\sum_{j=1}^N\biggl(\frac{\partial p_{ij}^n}{\partial \mbfy}(\mbfy^{(2)}-\mbfy^n)\biggl(1-\delta+\delta\frac{y_j^{(3)}}{\rho_j}\biggr)-\frac{\partial d_{ij}^n}{\partial \mbfy}(\mbfy^{(2)}-\mbfy^n)\frac{y_i^{(3)}}{\rho_i}\ \biggr)+\O(\Delta t^3)
\end{multline*}
for $i=1,\dots,N$.
Insertion of \eqref{eq:MPRKIIIaux10} and \eqref{eq:MPRKIIIaux16} shows
\begin{multline*}
 y_i^{(3)}=y_i^n + (a_{31}+a_{32})\Delta t\sum_{j=1}^N\biggl(p_{ij}^n\biggl(1-\delta+\delta\frac{y_j^{(3)}}{\rho_j}\biggr)-d_{ij}^n\frac{y_i^{(3)}}{\rho_i}\biggr)\\
 +a_{21}a_{32}\Delta t^2\frac{\partial (P_{i}^n-D_i^n)}{\partial \mbfy}(\mbfP^n-\mbfD^n)+\O(\Delta t^3)
\end{multline*}
for $i=1,\dots,N$.
Utilization of \eqref{eq:MPRK33c2} results in
\begin{multline}\label{eq:MPRKIIIaux14}
 b_2(y_i^{(2)}-y_i^n)+b_3(y_i^{(3)}-y_i^n)=\\
 \Delta t\sum_{j=1}^N\biggl(p_{ij}^n\biggl(b_2 a_{21}\biggl(1-\delta+\delta\frac{y_j^{(2)}}{\pi_j}\biggr)+b_3(a_{31}+a_{32})\biggl(1-\delta+\delta\frac{y_j^{(3)}}{\rho_j}\biggr)\biggr)\\
 -d_{ij}^n\biggl(b_2 a_{21}\frac{y_i^{(2)}}{\pi_i}+b_3(a_{31}+a_{32})\frac{y_i^{(3)}}{\rho_i}\biggr)\biggr)\\
 +\frac{\Delta t^2}6\frac{\partial (P_{i}^n-D_i^n)}{\partial \mbfy}(\mbfP^n-\mbfD^n)+\O(\Delta t^3)
\end{multline}
for $i=1,\dots,N$, since $b_3a_{21}a_{32}=1/6$ owing to \eqref{eq:RKorder6}.
Substitution of \eqref{eq:MPRKIIIaux10} and \eqref{eq:MPRKIIIaux11} into \eqref{eq:MPRKIIIaux14} in combination with \eqref{eq:MPRK33order3}  yields
\begin{multline*}
 b_2(y_i^{(2)}-y_i^n)+b_3(y_i^{(3)}-y_i^n)=
 \frac{\Delta t}2 (P_i^n-D_i^n)
 +\frac{\Delta t^2}6\frac{\partial (P_{i}^n-D_i^n)}{\partial \mbfy}(\mbfP^n-\mbfD^n)+\O(\Delta t^3)
\end{multline*}
for $i=1,\dots,N$.
Inserting this into \eqref{eq:MPRKIIIaux15} results in
\begin{multline*}
y_i^{n+1}=y_i^n+\Delta t(P_i^n-D_i^n)+\frac{\Delta t^2}2\frac{\partial(P_i^n-D_i^n)}{\partial\mbfy}(\mbfP^n-\mbfD^n)\\
+\frac{\Delta t^3 }6\sum_{k=1}^N\frac{\partial(P_i^n-D_i^n)}{\partial y_k}\frac{\partial (P_{k}^n-D_k^n)}{\partial \mbfy}(\mbfP^n-\mbfD^n)\\
+\frac{\Delta t^3}{6}(\mbfP^n-\mbfD^n)^T\mbfH_{P_i-D_i}^n(\mbfP^n-\mbfD^n)+\O(\Delta t^4)
\end{multline*}
for $i=1,\dots,N$. 
A comparison with \eqref{eq:MPRK33exact} completes the proof.
 \end{proof}

The following theorem defines a family of third order MPRK schemes. 
It is based on the idea to use a second order MPRK22($\alpha$) scheme of \cite{KopeczMeister2017} to compute the PWDs $\sigma_i$, as condition \eqref{eq:MPRK33order4} of Theorem~\ref{thm:MPRK33order} shows that $\sigma_i$ must be a second order approximation of $y_i(t^{n+1})$ for $i=1,\dots,N$.
\begin{thm}\label{thm:MPRK43PWD}
 Given an explicit three-stage third order Runge-Kutta scheme with non-negative weights,
 the MPRK scheme \eqref{eq:MPRK33} is of third order, if we choose
 \begin{subequations}\label{eq:PWD43}
 \begin{align}\label{eq:PWD43pi}
  \pi_i&=y_i^n,\quad i=1,\dots,N\\
  \rho_i&=y_i^n\left(\frac{y_i^{(2)}}{y_i^n}\right)^{\!\!1/p},\quad p=3a_{21}(a_{31}+a_{32})b_3,\quad i=1,\dots,N, \label{eq:PWD43rhop}\\
   &\begin{aligned}\label{eq:PWD43mu}
     \mathllap{\mu_i} &= y_i^n\left(\frac{y_i^{(2)}}{y_i^n}\right)^{1/q},\quad q=a_{21}, \quad i=1,\dots,N,
   \end{aligned}\\ 
     &\begin{multlined}[10cm] \label{eq:PWD43sigma}
    \mathllap{\sigma_i} = y_i^n + \Delta t\sum_{j=1}^N\biggl( \left(\beta_1 p_{ij}(\mbfy^{(1)})
+\beta_2 p_{ij}(\mbfy^{(2)})\right)\frac{\sigma_j}{\mu_j}\\
- \left(\beta_1 d_{ij}(\mbfy^{(1)})+ \beta_2 d_{ij}(\mbfy^{(2)})\right)\frac{\sigma_i}{\mu_i}\biggr),\quad i=1,\dots,N.
  \end{multlined} 
\end{align}
 with $\beta_1=1-\beta_2$ and $\beta_2=1/(2a_{21})$.
\end{subequations}
\end{thm}
\begin{proof}
We need to verify that the choice of PWDs \eqref{eq:PWD43} satisfies conditions \eqref{eq:MPRK33order} of Theorem~\ref{thm:MPRK33order}.
Therefore, we make repeatedly use of the statements of the proof of Theorem~\ref{thm:MPRK33order}.
We also use Newton's generalized binomial theorem\footnote{The theorem can be deduced from the binomial series $\sum_{k=0}^\infty \binom{s}{k}x^k=(1+x)^s$, which is convergent for $\abs{x}<1$. See for instance \cite{howie2001}.}, which states that
\begin{equation}\label{eq:binomseries}
(x+y)^s = \sum_{k=0}^\infty\binom{s}{k}x^{s-k}y^k
\end{equation}
with
\begin{equation*}
 \binom{s}{k}=\frac{s(s-1)\dots (s-k+1)}{k!},\quad \binom{s}{0}=1
\end{equation*}
holds true for $s\in\mathbb R$, if $x>0$ and $\abs{y/x}<1$.
The theorem implies
\begin{equation}\label{eq:binomser_impla}
 \bigl(y_i^n+\O(\Delta t)\bigr)^s=(y_i^n)^s+\O(\Delta t)
\end{equation}
and 
\begin{equation}\label{eq:binomser_implb}
 \bigl(y_i^n+\eta\Delta t+\O(\Delta t^2)\bigr)^s=(y_i^n)^s+ s(y_i^n)^{s-1}\eta\Delta t+\O(\Delta t^2)
\end{equation}
for $s,\eta\in\mathbb R$, since $y_i^n>0$.

First, we note that condition \eqref{eq:MPRK33order1} is clearly satisfied by \eqref{eq:PWD43pi}. 
This allows us to conclude 
\begin{equation}\label{eq:PWD43aux1}
y_i^{(2)}=y_i^n+a_{21}\Delta t(P_i^n-D_i^n)+\O(\Delta t^2)
\end{equation}
for $i=1,\dots,N$, along the same lines as in \eqref{eq:MPRKIIIaux10} of Theorem~\ref{thm:MPRK33order}.
Introducing this into \eqref{eq:PWD43rhop} shows 
\begin{equation*}
\rho_i=(y_i^n)^{1/p-1}\bigl(y_i^n+\O(\Delta t)\bigr)^{1/p}\eq^{\eqref{eq:binomser_impla}}y_i^n+\O(\Delta t)
\end{equation*}
for $i=1,\dots,N$. 
Thus, condition \eqref{eq:MPRK33order2} holds true as well.

Next, we verify condition \eqref{eq:MPRK33order3}. 
From \eqref{eq:binomser_implb} and \eqref{eq:PWD43aux1} we find
\begin{equation*}
\rho_i=y_i^n + \frac{\Delta t(P_i^n-D_i^n)}{3(a_{31}+a_{32})b_3}+\O(\Delta t^2)
\end{equation*}
for $i=1,\dots,N$.
Defining $f(\Delta t)=1/(\xi+\Delta t \eta)$ for some constants $\xi$ and $\eta$, we can conclude $f(\Delta t)=f(0) + f'(0)\Delta t + \O(\Delta t^2)=1/\xi -\eta/\xi^2+\O(\Delta t^2)$, and hence,
\begin{equation*}
\frac{1}{\rho_i}=\frac{1}{y_i^n}-\frac{\Delta t(P_i^n-D_i^n)}{3(a_{31}+a_{32})b_3(y_i^n)^2}+\O(\Delta t^2)
\end{equation*}
for $i=1,\dots,N$.
Consequently,
\begin{multline*}
\frac{y_i^n + (a_{31}+a_{32})\Delta t(P_i^n-D_i^n)}{\rho_i}\\
=\bigl(y_i^n+(a_{31}+a_{32})\Delta t(P_i^n-D_i^n)\bigr)\left(\frac{1}{y_i^n}-\frac{\Delta t(P_i^n-D_i^n)}{3(a_{31}+a_{32})b_3(y_i^n)^2}+\O(\Delta t^2)\right)\\=
1+\frac{\Delta t(P_i^n-D_i^n)}{y_i^n}\left(a_{31}+a_{32}-\frac{1}{3(a_{31}+a_{32})b_3}\right)+\O(\Delta t^2)
\end{multline*}
for $i=1,\dots,N$.
Substituting this and \eqref{eq:PWD43pi} into condition \eqref{eq:MPRK33order3} shows
\begin{multline*}
 b_2a_{21}\frac{y_i^n+a_{21}\Delta t(P_i^n-D_i^n)}{\pi_i}+b_3(a_{31}+a_{32})\frac{y_i^n+(a_{31}+a_{32})\Delta t(P_i^n-D_i^n)}{\rho_i}\\
 =\underbrace{b_2a_{21}+b_3(a_{31}+a_{32})}_{=1/2}+\frac{\Delta t(P_i^n-D_i^n)}{y_i^n}\biggl(\underbrace{b_2a_{21}^2+b_3(a_{31}+a_{32})^2}_{=1/3}-\frac13\biggr)+\O(\Delta t^2)\\
 =\frac12+\O(\Delta t^2)
\end{multline*}
for $i=1,\dots,N$.
Hence, condition \eqref{eq:MPRK33order3} holds true.
Finally, \eqref{eq:MPRK33c2} and \eqref{eq:PWD43sigma} with PWDs \eqref{eq:PWD43mu} form the MPRK22($a_{21}$) scheme of \cite{KopeczMeister2017}.
As this is a second order scheme, condition \eqref{eq:MPRK33order4} is satisfied as well.
\end{proof}
The family of schemes introduced in Theorem~\ref{thm:MPRK43PWD} can be written in the form
\begin{subequations}\label{eq:MPRK43}
\begin{align}
  &\begin{aligned}
    \mathllap{y_i^{(1)}} &= y_i^n,
  \end{aligned}\\ \
  &\begin{aligned} 
    \mathllap{y_i^{(2)}} &= y_i^n + a_{21}\Delta t\sum_{j=1}^N\biggl( p_{ij}(\mbfy^{(1)})(1-\delta)+p_{ij}(\mbfy^{(1)})\frac{y_j^{(2)}}{y_j^n}\delta-d_{ij}(\mbfy^{(1)})\frac{y_i^{(2)}}{y_i^n}\biggr),
  \end{aligned}\\ 
  &\begin{aligned} 
    \mathllap{y_i^{(3)}} &= y_i^n + \Delta t\sum_{j=1}^N\biggl( \left(a_{31} p_{ij}(\mbfy^{(1)})+a_{32} p_{ij}(\mbfy^{(2)})\right)(1-\delta)\\
    &\qquad+\frac{\left(a_{31} p_{ij}(\mbfy^{(1)})+a_{32} p_{ij}(\mbfy^{(2)})\right)y_j^{(3)}\delta}{(y_j^{(2)})^{1/p}(y_j^{n})^{1/p-1}}-\frac{\left( a_{31} d_{ij}(\mbfy^{(1)})+ a_{32} d_{ij}(\mbfy^{(2)})\right)y_i^{(3)}}{(y_i^{(2)})^{1/p}(y_i^{n})^{1/p-1}}\biggr).
  \end{aligned}  \label{eq:MPRK43c}\\
    &\begin{aligned} 
    \mathllap{\sigma_i} &= y_i^n + \Delta t\sum_{j=1}^N\biggl( \frac{\left(\beta_1 p_{ij}(\mbfy^{(1)})
+\beta_2 p_{ij}(\mbfy^{(2)})\right)\sigma_j}{(y_j^{(2)})^{1/q}(y_j^{n})^{1/q-1}}- \frac{\left(\beta_1 d_{ij}(\mbfy^{(1)})+ \beta_2 d_{ij}(\mbfy^{(2)})\right)\sigma_i}{(y_i^{(2)})^{1/q}(y_i^{n})^{1/q-1}}\biggr).
  \end{aligned}\label{eq:MPRK43d}  \\
  &\begin{aligned} 
    \mathllap{y_i^{n+1}} &= y_i^n + \Delta t\sum_{j=1}^N\biggl( \left(b_1 p_{ij}(\mbfy^{(1)})+b_2 p_{ij}(\mbfy^{(2)})+b_3 p_{ij}(\mbfy^{(3)})\right)\frac{y_j^{n+1}}{\sigma_j}\biggr.\\
         &\biggl.\qquad\qquad\qquad\qquad- \left(b_1 d_{ij}(\mbfy^{(1)})- b_2 d_{ij}(\mbfy^{(2)})- b_3 d_{ij}(\mbfy^{(3)})\right)\frac{y_i^{n+1}}{\sigma_i}\biggr)
  \end{aligned}  
\end{align}
\end{subequations}
with $p=3a_{21}(a_{31}+a_{32})b_3$, $q=a_{21}$, $\beta_2=1/(2a_{21})$ and $\beta_1=1-\beta_2$ for $i=1,\dots,N$.
We denote the members of this family, which derive from case I in Lemma~\ref{lem:RK3pos}, by MPRK43I($\alpha$,$\beta$) if $\delta=1$ and by MPRK43Incs($\alpha$,$\beta$) if $\delta=0$. 
If the method comes from case II in Lemma~\ref{lem:RK3pos}, we denote it by MPRK43II($\gamma$) if $\delta=1$ and by MPRK43IIncs($\gamma$) if $\delta=0$.

To our knowledge, this is the first time that third order MPRK schemes are presented.
A third order Patankar type scheme based on a BDF method was presented in \cite{FormaggiaScotti2011}.

As the schemes \eqref{eq:MPRK43} incorporate the MPRK22($a_{21}$) scheme, we must restrict $a_{21}$ to $a_{21}\geq 1/2$.
Hence, the permissible Runge-Kutta parameters are given by the Butcher tableaus of Lemma~\ref{lem:RK3pos} with the additional restriction $\alpha\geq 1/2$ in case I.

The MPRK scheme \eqref{eq:MPRK43} can be understood as a four stage MPRK scheme with corresponding Butcher tableau
\[
\begin{array}{c|cccc}
 0        &                    & \\
 a_{21} & a_{21}\\
 a_{31}+a_{32} & a_{31}           & a_{32} & \\
 1 & \beta_1           & \beta_2 & \\\hline
          & b_1 & b_2 & b_3 & 0
\end{array}.
\]
The extra stage to compute the PWDs $\sigma_i$ requires no additional function evaluations, but nevertheless an additional linear system needs to be solved.
It is a future concern to prove or disprove, whether the construction of a third order three stage MPRK scheme is possible.

In the numerical experiments of Section~\ref{sec:numres} we consider six specific MPRK43 schemes.
Choosing $\alpha=1$ and $\beta=1/2$ in case I of Lemma~\ref{lem:RK3pos} yields the Butcher tableau
\[
\begin{array}{c|ccc}
 0        &                    & \\
 1 & 1\\
 1/2 & 1/4           & 1/4 & \\\hline
          & 1/6 & 1/6 & 2/3
\end{array}.
\]
The corresponding scheme MPRK43I($1$,$1/2$) is the only MPRK43 scheme with $p=q=1$ and it uses the original MPRK22(1) scheme from \cite{BDM2003}, which is based on Heun's method, to compute the PWDs $\sigma_i$.
Setting $\alpha=1/2$ and $\beta=3/4$ in case I of Lemma~\ref{lem:RK3pos}, we obtain the Butcher tableau
\[
\begin{array}{c|ccc}
 0        &                    & \\
 1/2 & 1/2\\
 3/4 & 0  & 3/4 & \\\hline
     & 2/9 & 1/3 & 4/9
\end{array}
\]
and the method MPRK43I($1/2$,$3/4$) with $p=q=1/2$. This method uses the MPRK22($1/2$) scheme, which is adapted from the midpoint method, to compute the PWDs $\sigma_i$.
Case II of Lemma~\ref{lem:RK3pos} with $\gamma=1/2$ provides the Butcher tableau 
\[
\begin{array}{c|ccc}
 0        &                    & \\
 2/3 & 2/3\\
 2/3 & 1/6  & 1/2 & \\\hline
     & 1/4 & 1/4 & 1/2
\end{array}.
\]
The associated MPRK scheme MPRK43II($2/3$) with $p=q=2/3$.
It employs Ralston's method MPRK22(2/3) to calculate the PWDs $\sigma_i$.
Besides the above three schemes, we also use their counterparts with $\delta=0$ in the numerical experiments of Section~\ref{sec:numres}. 

Of course, many other schemes are members of the family \eqref{eq:MPRK43}.
Here we made the restriction $p=q$ to allow for the same PWDs in \eqref{eq:MPRK43c} and \eqref{eq:MPRK43d}.
We also selected schemes, which use MPRK22($\alpha$) schemes to compute the PWDs $\sigma_i$, that were investigated in \cite{KopeczMeister2017}.

The numerical experiments, presented in Section~\ref{sec:numres}, will confirm the third order accuracy of the MPRK43 schemes. 
Additionally, numerical solutions of the Robertson problems will show that these schemes have the ability to integrate stiff PDS.
\section{Test problems}\label{sec:testcases}
For our numerical experiments, we consider the same test cases as in \cite{KopeczMeister2017}.
A simple linear test problem for which the analytical solution is known, two non-stiff nonlinear test problems and the stiff Robertson problem.

\subsection*{Linear test problem}
The simple linear test case is given by
\begin{equation}\label{eq:lintest}
 y_1'(t) = y_2(t) - a y_1(t),\quad y_2'(t) = a y_1(t) - y_2(t),
\end{equation}
with a constant parameter $a$ and initial values $y_1(0) = y_1^0$ and $y_2(0)=y_2^0$. 
We can write the right hand side in the form \eqref{eq:pijdij} with
\begin{align*}
p_{12}(\mbfy) &= y_2, & p_{21}(\mbfy) &= a y_1,\\
d_{12}(\mbfy) &= a y_1, & d_{21}(\mbfy) &= y_2,
\end{align*}
and $p_{ii}(\mbfy)=d_{ii}(\mbfy)=0$ for $i=1,2$.
The system describes exchange of mass between to constituents. The analytical solution is
\[
y_1(t) = (1+c\exp(-(a+1)t))y_1^\infty
\]
with the asymptotic solution
\[
y_1^\infty = \frac{y_1^0+y_2^0}{a+1},\quad c = \frac{y_1^0}{y_1^\infty}-1.
\]
The system is conservative and we get
\[
y_2(t) = y_1^0+y_2^0 - y_1(t).
\]
In the numerical simulations of Section~\ref{sec:numres} we use $a=5$ and initial values $y_1^0=0.9$ and $y_2^0=0.1$.
The solution is approximated on the time interval $[0,1.75]$.
\subsection*{Nonlinear test problem}
The non-stiff nonlinear test problem reads
\begin{equation}\label{eq:nonlintest}
 \begin{split}
  y_1'(t) &= -\frac{y_1(t)y_2(t)}{y_1(t)+1},\\
  y_2'(t) &= \frac{y_1(t)y_2(t)}{y_1(t)+1}-a y_2(t),\\
  y_3'(t) &= a y_2(t),
 \end{split}
\end{equation}
with initial conditions $y_i(0)=y_i^0$ for $i=1,2,3$.
To express the right hand side in the form \eqref{eq:pijdij} we can use
\begin{align*}
 p_{21}(\mbfy)=d_{21}(\mbfy)= \frac{y_1y_2}{y_1+1},\quad p_{32}(\mbfy)=d_{23}(\mbfy)=a y_2,
\end{align*}
and $p_{ij}(\mbfy)=d_{ij}(\mbfy)=0$ for all other combinations of $i$ and $j$.

The system represents a biogeochemical model for the description of an algal bloom, that transforms nutrients ($y_1$) via phytoplankton ($y_2$) into detritus ($y_3$).
In the numerical simulations of Section~\ref{sec:numres} we use the initial conditions $y_1^0=9.98$, $y_2^0=0.01$ and $y_3^0=0.01$.
The solution is approximated on the time interval $[0,30]$.
\subsection*{Original Brusselator test problem}
As another non-stiff nonlinear test case we consider the original Brusselator problem \cite{LefeverNicolis1971, HNW1993}
\begin{equation}\label{eq:brusselator}
 \begin{split}
  y_1'(t) &= -k_1 y_1(t),\\
  y_2'(t) &= -k_2 y_2(t) y_5(t),\\
  y_3'(t) &= k_2 y_2(t) y_5(t),\\
  y_4'(t) &= k_4 y_5(t),\\
  y_5'(t) &= k_1 y_1(t) - k_2 y_2(t) y_5(t) +k_3 y_5(t)^2y_6(t)-k_4 y_5(t),\\
  y_6'(t) &= k_2 y_2(t) y_5(t) - k_3 y_5(t)^2 y_6(t),
 \end{split}
\end{equation}
with constant parameters $k_i$ and initial values $y_i(0)=y_i^0$ for $i=1,\dots,6$.
The system can be written in the form \eqref{eq:pijdij}, setting
\begin{align*}
 p_{32}(\mbfy) &=d_{23}(\mbfy)= k_2 y_2 y_5, & p_{45}(\mbfy)&=d_{54}(\mbfy)=k_4 y_5, & p_{51}(\mbfy)&=d_{15}(\mbfy)=k_1 y_1,\\
 p_{56}(\mbfy)&=d_{65}(\mbfy)=k_3 y_5^2 y_6, & p_{65}(\mbfy)&=d_{56}(\mbfy)=k_2 y_2 y_5,
\end{align*}
and $p_{ij}(\mbfy)=d_{ji}(\mbfy)=0$ for all other combinations of $i$ and $j$.

In the numerical simulations of Section~\ref{sec:numres} we set $k_i=1$ for $i=1,\dots,6$ and the initial values $y_1(0)=y_2(0)=10$, $y_3(0)=y_4(0)=\mathtt
{eps}\approx2.2204\cdot10^{-16}$, and $y_5(0)=y_6(0)=0.1$.
The time interval of interest is $[0,6]$.
\subsection*{Robertson test problem}
To demonstrate the practicability of MPRK schemes in the case of stiff systems, we apply the schemes to the Robertson test case, which is given by
\begin{equation}\label{eq:robertson}
 \begin{split}
  y_1'(t) &= 10^4 y_2(t)y_3(t) - 0.04 y_1(t),\\
  y_2'(t) & = 0.04 y_1(t) - 10^4 y_2(t) y_3(t) - 3\cdot 10^7 y_2(t)^2,\\
  y_3'(t) &= 3\cdot 10^7 y_2(t)^2,
 \end{split}
\end{equation}
with initial values $y_i(0)=y_i^0$ for $i=1,2,3$.
For this problem the production and destruction rates \eqref{eq:pijdij} are given by
\[p_{12}(\mbfy)=d_{21}(\mbfy)=10^4y_2y_3,\quad p_{21}(\mbfy)=d_{12}(\mbfy)=0.04y_1,\quad p_{32}(\mbfy)=d_{23}(\mbfy)=3\cdot10^7 y_2,\]
and $p_{ij}(\mbfy)=d_{ij}(\mbfy)=0$ for all other combinations of $i$ and $j$.

We use the initial values $y_1(0)=y_1^0=1-2 \mathtt
{eps}$ and $y_2^0=y_3^0=\mathtt
{eps}\approx2.2204\cdot10^{-16}$ in the numerical simulations of Section~\ref{sec:numres}.

In this problem the reactions take place on very different time scales, the time interval of interest is $[10^{-6},10^{10}]$. Therefore, a constant time step size is not appropriate.
In the numerical simulations we use $\Delta t_i = 4^{i-1}\Delta t_0$ with $\Delta t_0=10^{-6}$ in the $i$th time step.
The small initial time step size $\Delta t_0$ is chosen to obtain an adequate resolution of $y_2$.

\section{Numerical results}\label{sec:numres}
\begin{figure}
\centering
\begin{subfigure}{.49\textwidth}
\centering
\includegraphics[width=\textwidth]{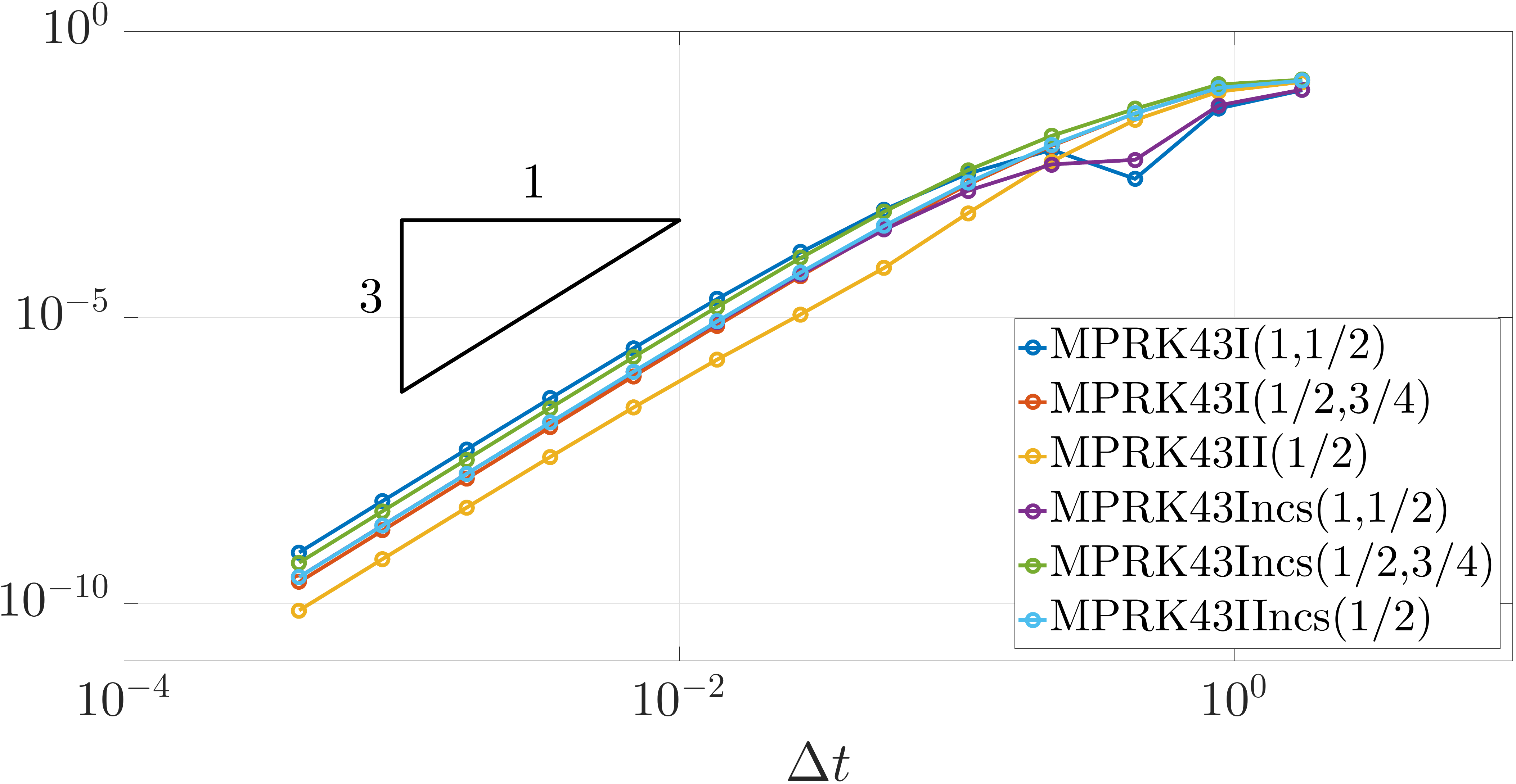}
\caption{Linear test problem \eqref{eq:lintest}}
\label{fig:testorder1}
\end{subfigure}
\hfill
\begin{subfigure}{.49\textwidth}
\centering
\includegraphics[width=\textwidth]{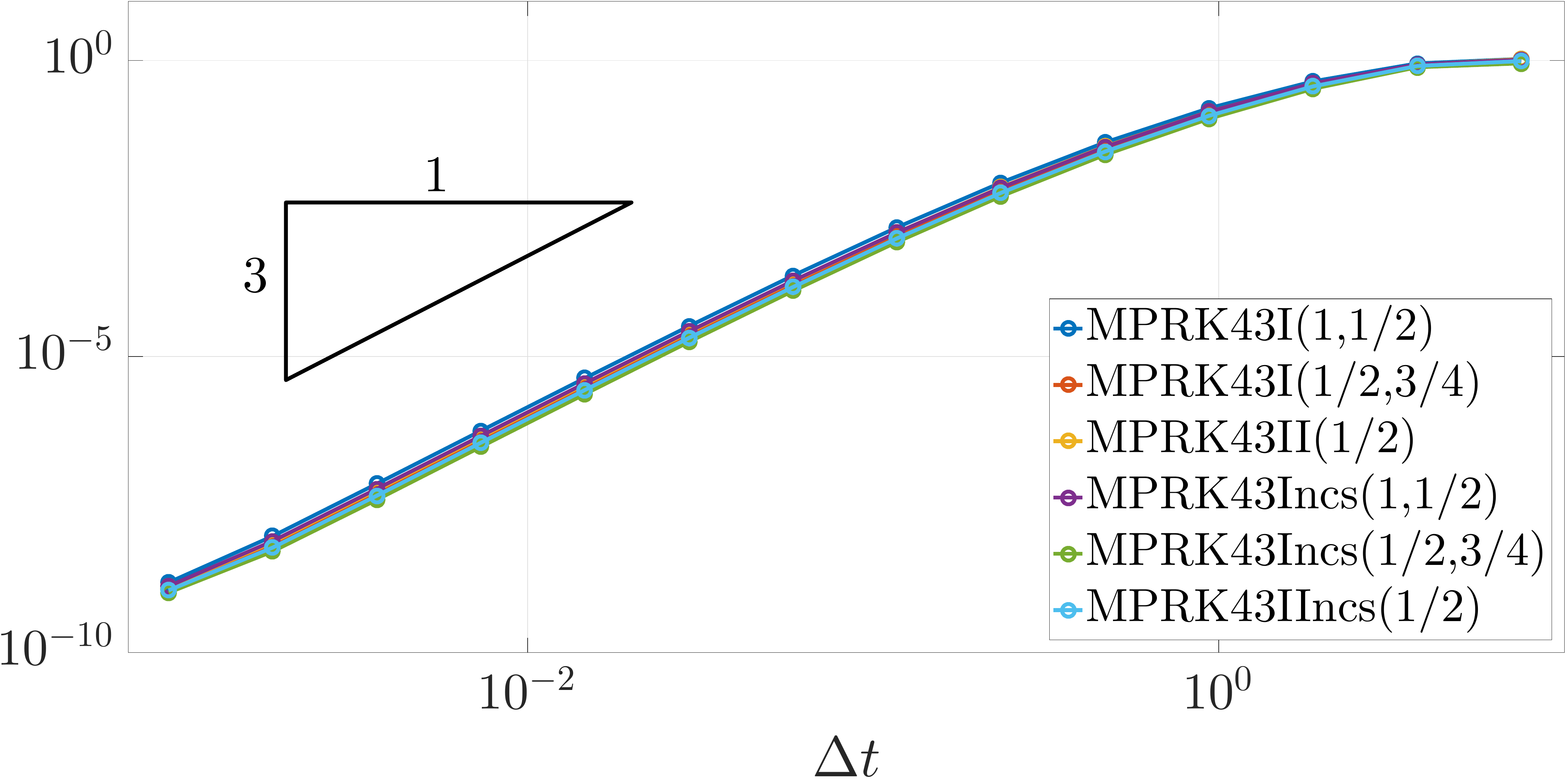}
\caption{Nonlinear test problem \eqref{eq:nonlintest}}
\label{fig:testorder2}
\end{subfigure}
\par\vspace{.5\baselineskip}
\begin{subfigure}{.49\textwidth}
\centering
\includegraphics[width=\textwidth]{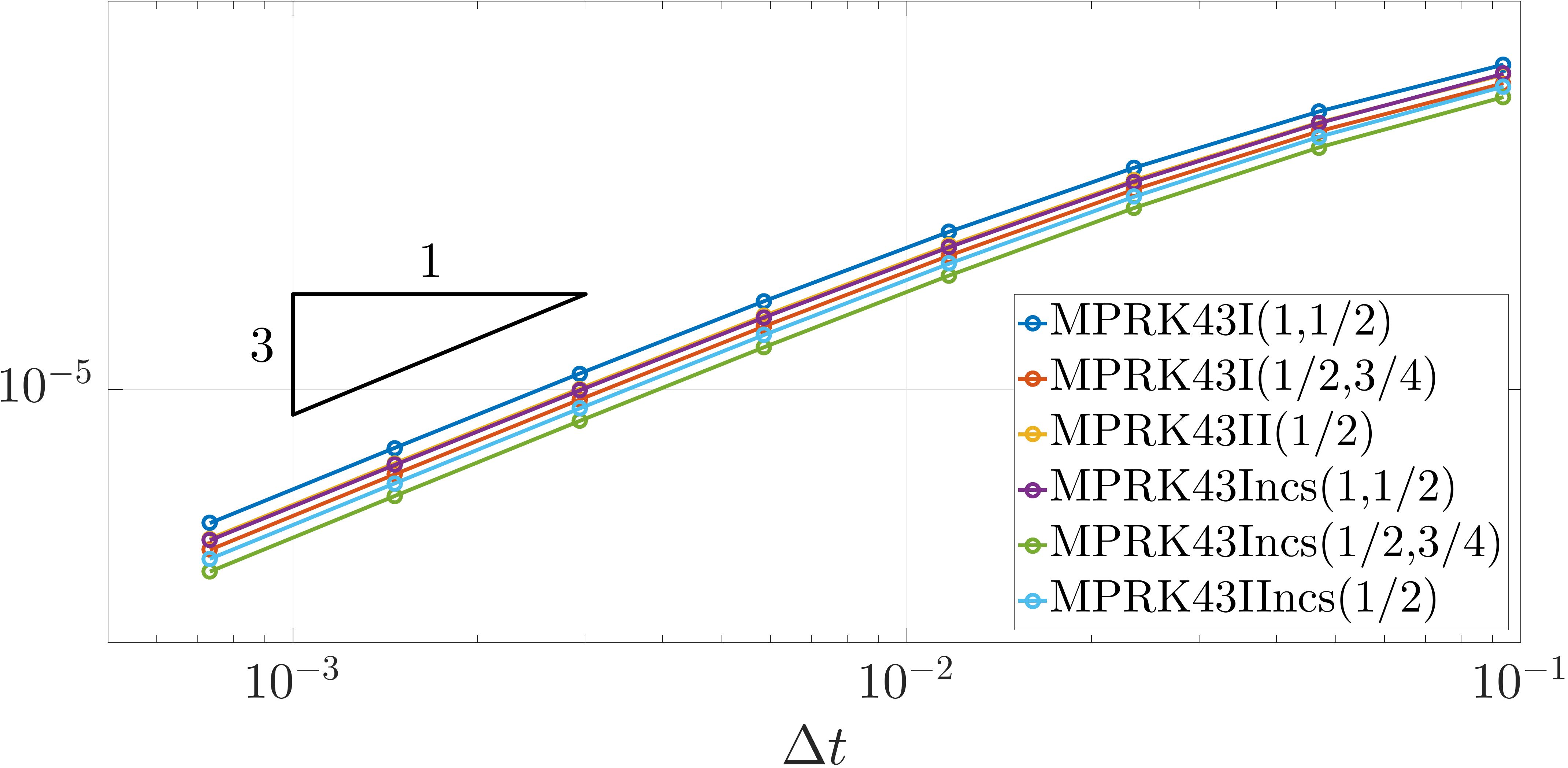}
\caption{Brusselator \eqref{eq:brusselator}}
\label{fig:testorder3}
\end{subfigure}
\caption{Error plots of various MPRK43 schemes.}
\label{fig:testorder}
\end{figure}
\begin{figure}[htb]
\centering
\begin{subfigure}{.49\textwidth}
\centering
 \includegraphics[width=\textwidth]{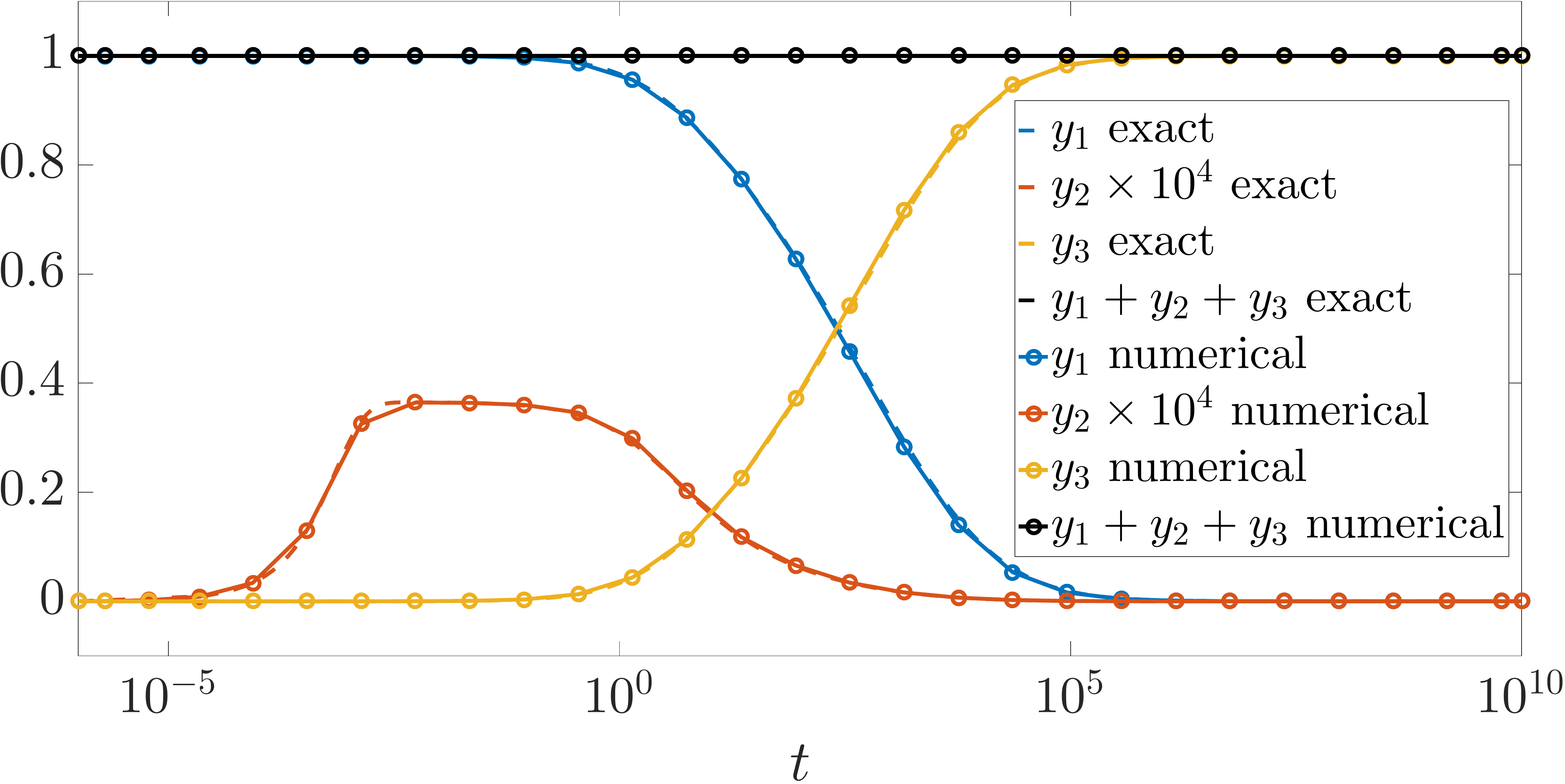}
 \caption{MPRK43I($1,1/2$)}
 \end{subfigure}
 \hfill
\begin{subfigure}{.49\textwidth}
\centering
 \includegraphics[width=\textwidth]{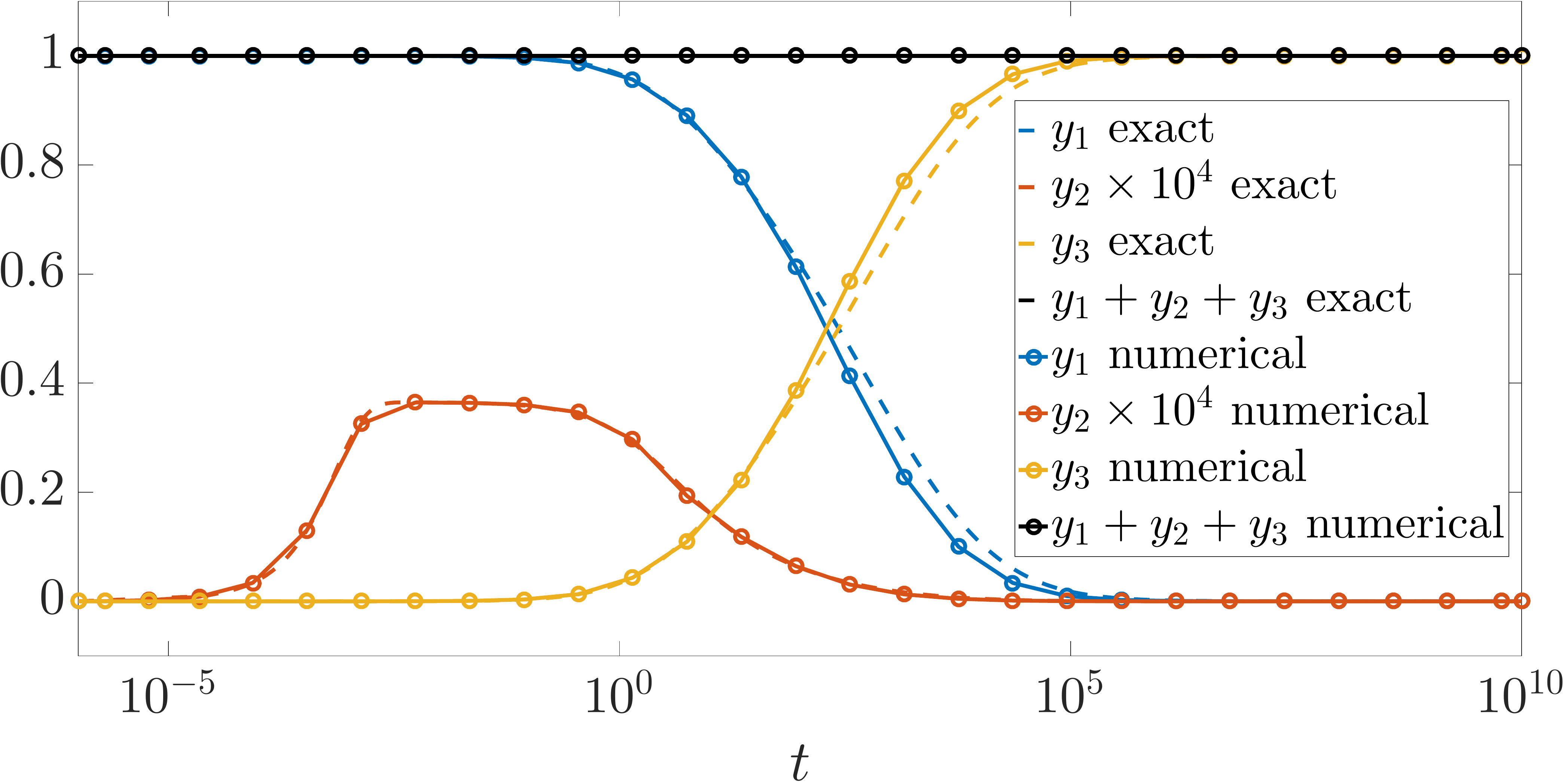}
 \caption{MPRK43Incs($1,1/2$)}
 \end{subfigure} 
 \par
\vspace{2\baselineskip} 
 \begin{subfigure}{.49\textwidth}
\centering
 \includegraphics[width=\textwidth]{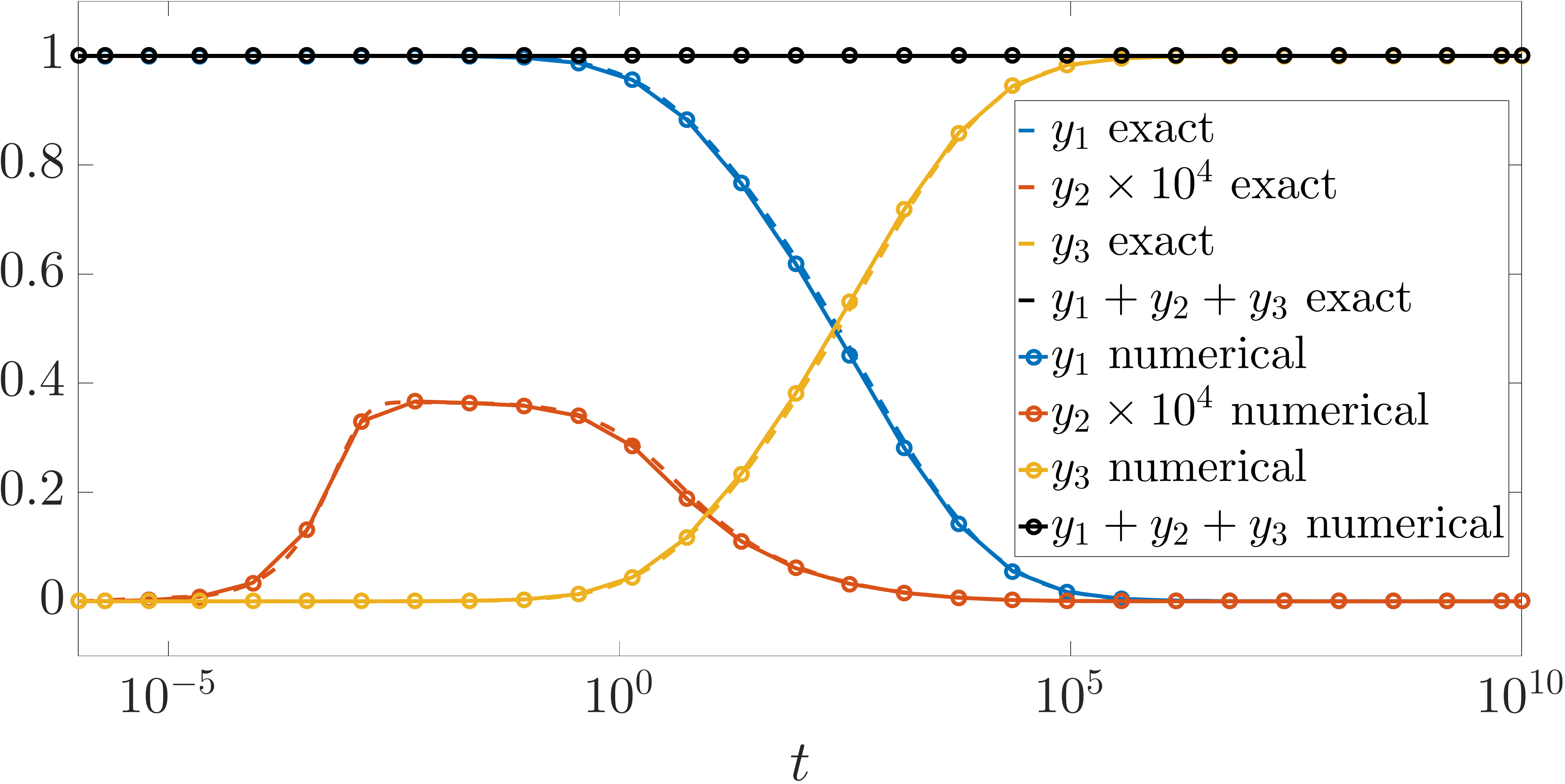}
 \caption{MPRK43I($1/2,3/4$)}
 \end{subfigure}
\hfill
 \begin{subfigure}{.49\textwidth}
\centering
 \includegraphics[width=\textwidth]{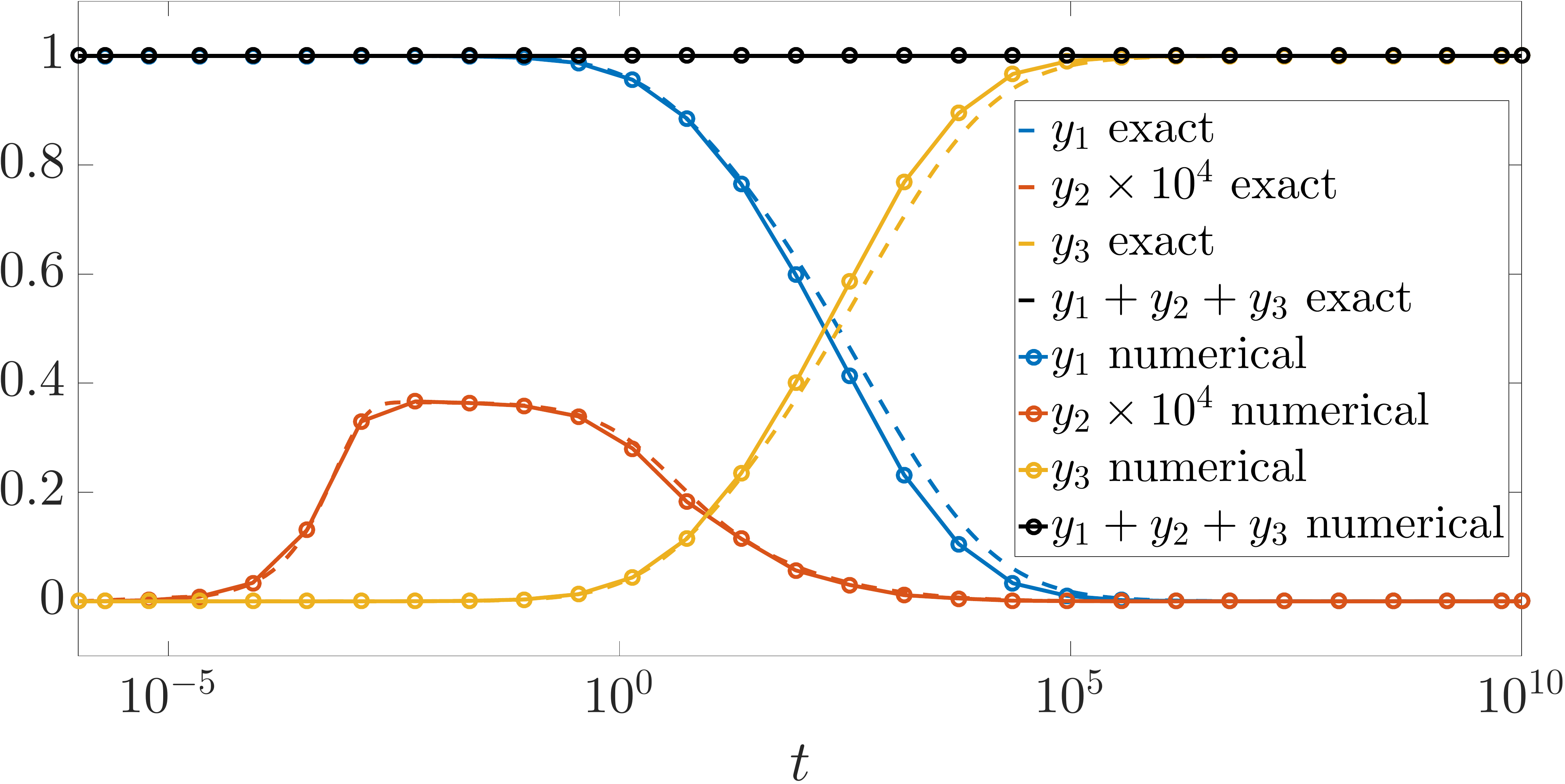}
 \caption{MPRK43Incs($1/2,3/4$)}
 \end{subfigure}
 \par\vspace{2\baselineskip} 
 \begin{subfigure}{.49\textwidth}
\centering
 \includegraphics[width=\textwidth]{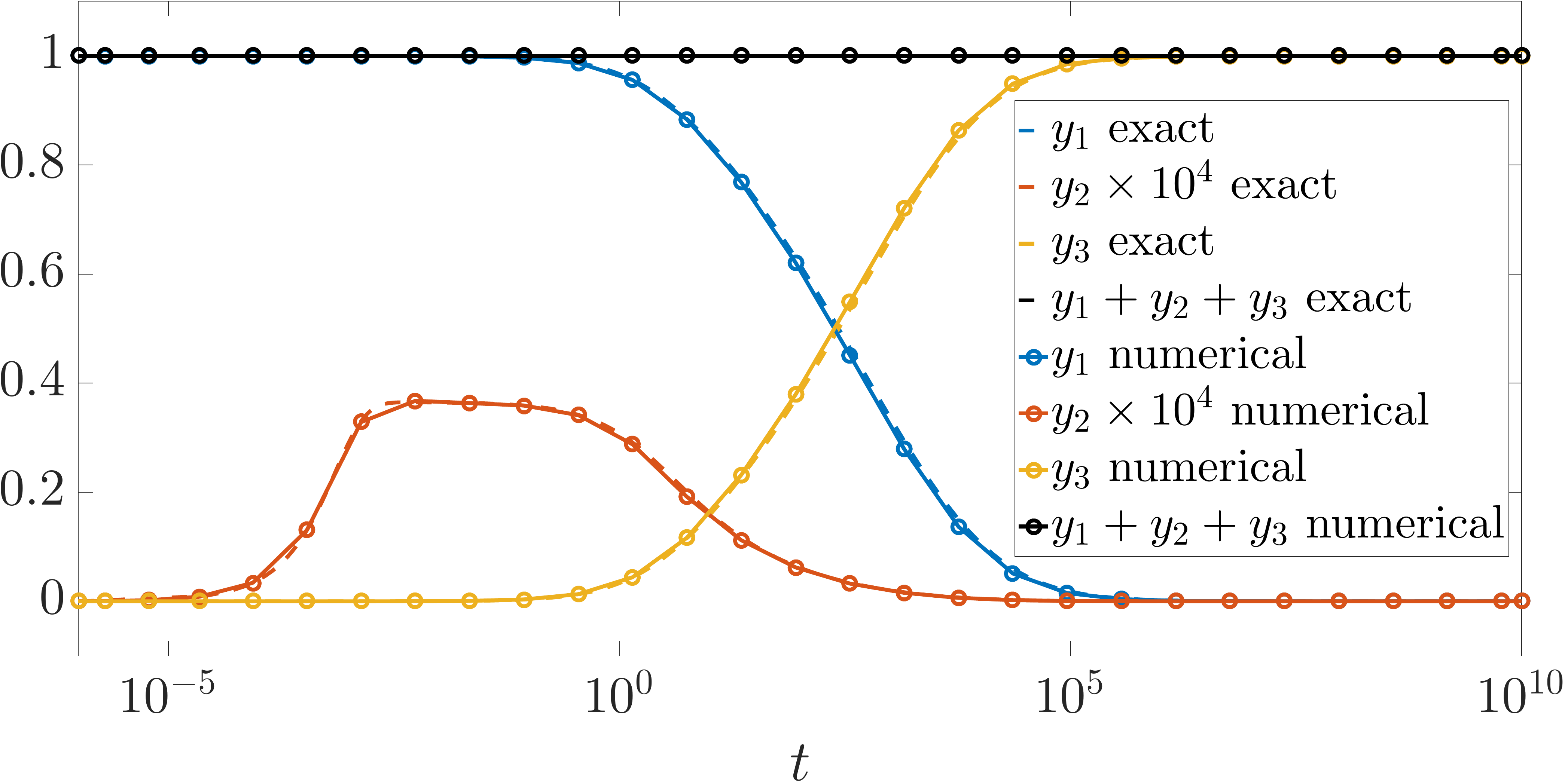}
 \caption{MPRK43II($1/2$)}
 \end{subfigure}
\hfill
 \begin{subfigure}{.49\textwidth}
\centering
 \includegraphics[width=\textwidth]{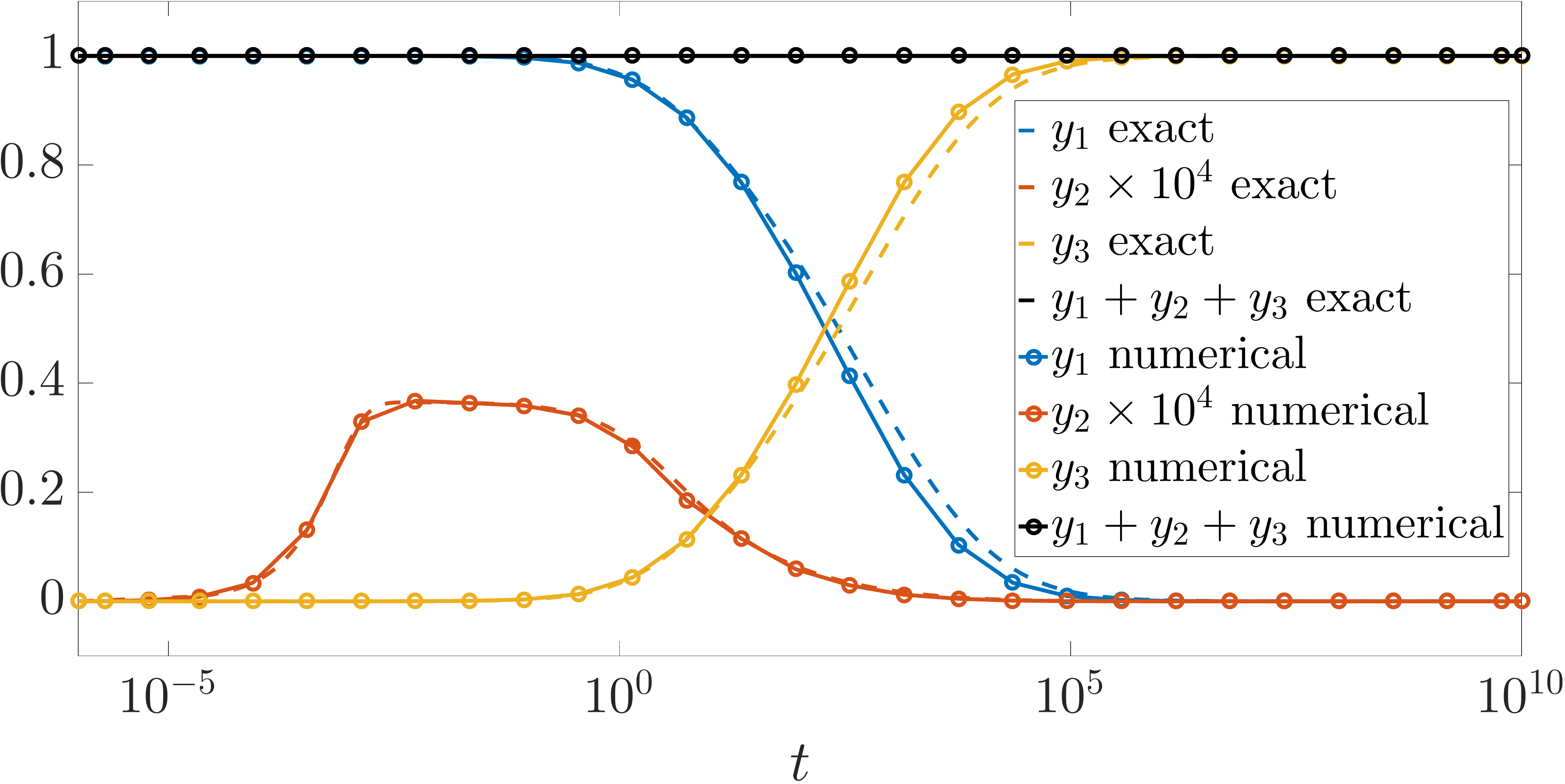}
 \caption{MPRK43IIncs($1/2$)}
 \end{subfigure}
 \caption{Numerical solutions of the Robertson problem \eqref{eq:robertson} for different MPRK43 schemes.}
 \label{fig:robertson}
\end{figure}

\begin{figure}[htb]
\centering
\begin{subfigure}{.49\textwidth}
\centering
 \includegraphics[width=\textwidth]{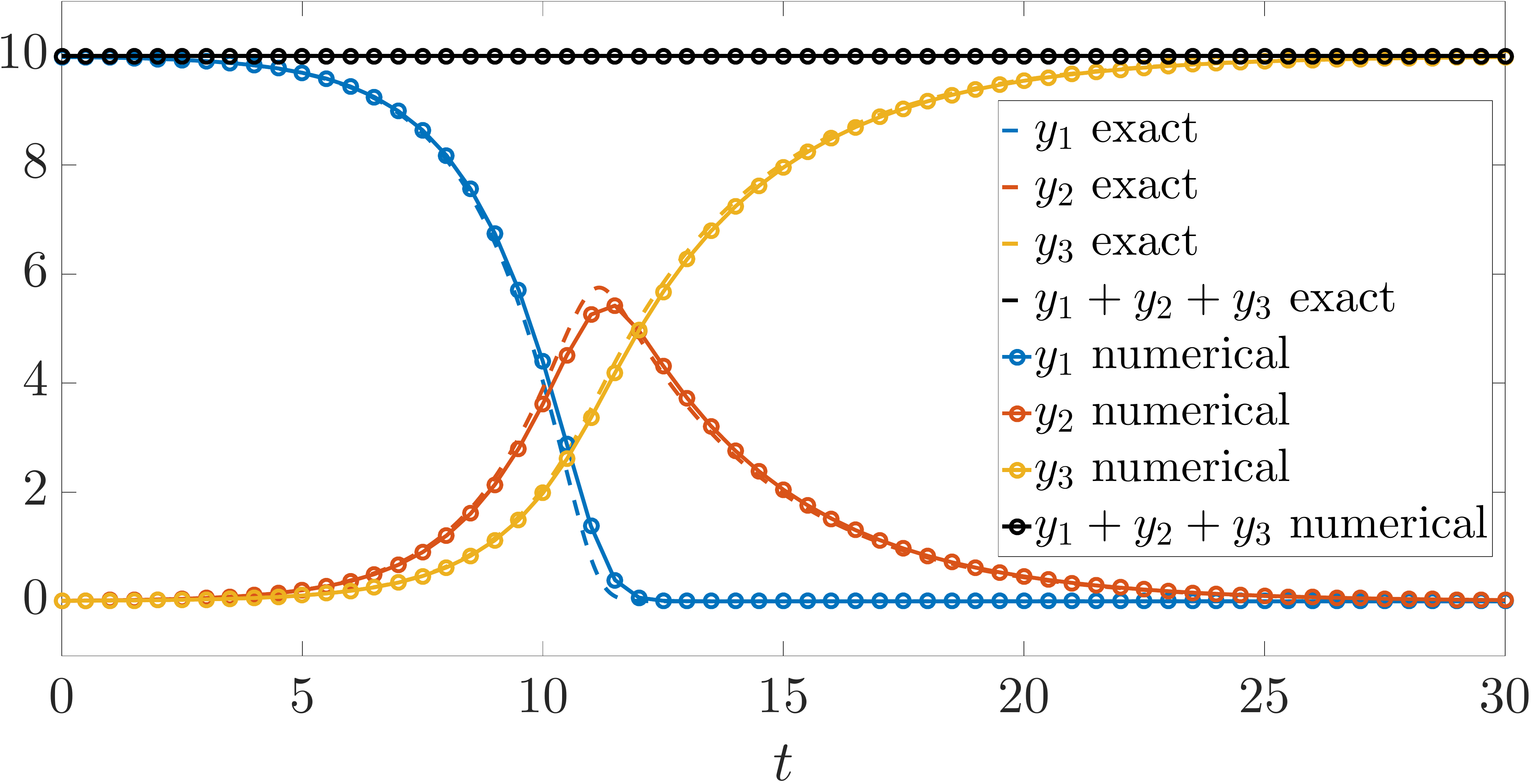}
 \caption{MPRK43I($1,1/2$)}
 \end{subfigure}
 \hfill
\begin{subfigure}{.49\textwidth}
\centering
 \includegraphics[width=\textwidth]{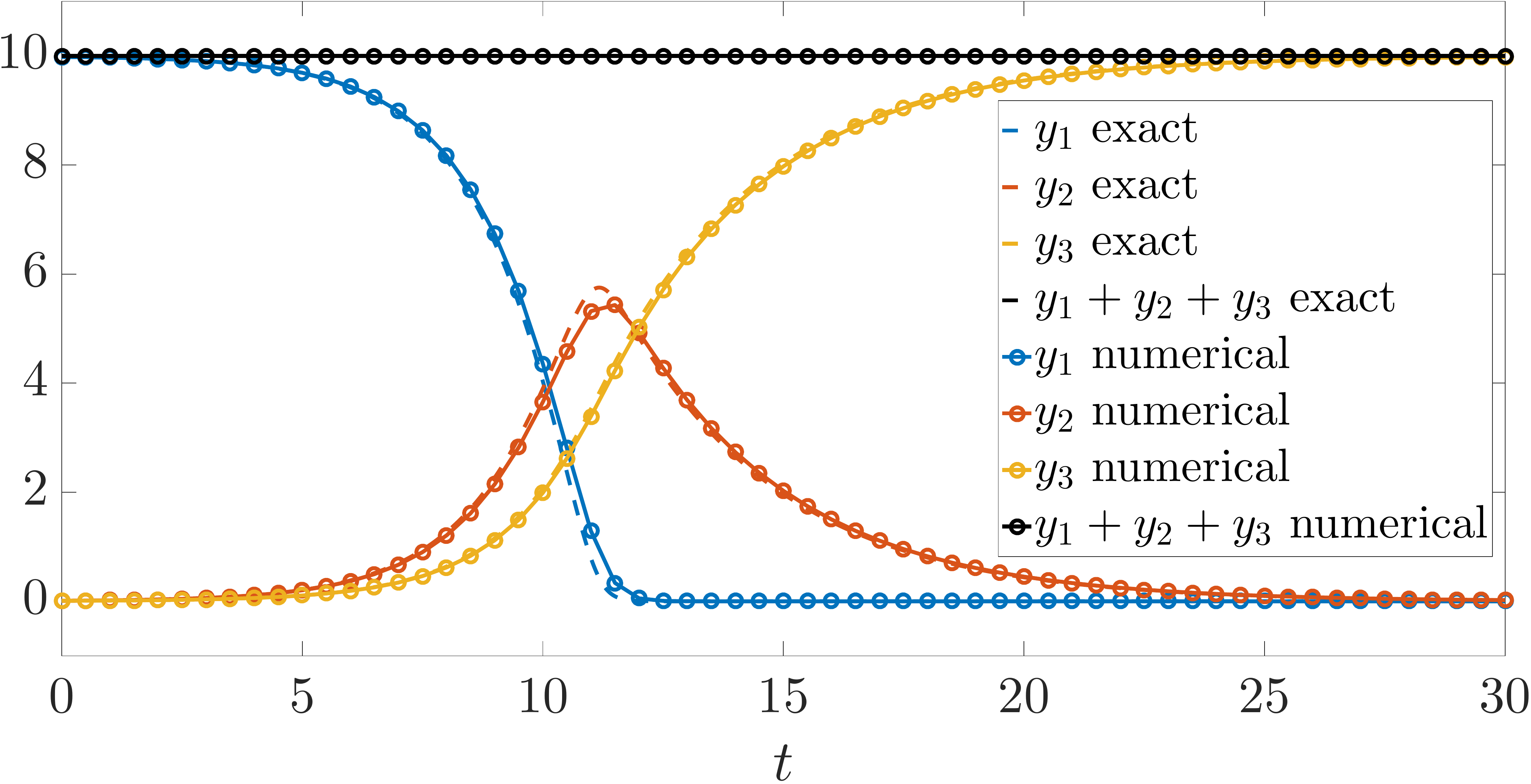}
 \caption{MPRK43Incs($1,1/2$)}
 \end{subfigure} 
 \par
\vspace{2\baselineskip} 
 \begin{subfigure}{.49\textwidth}
\centering
 \includegraphics[width=\textwidth]{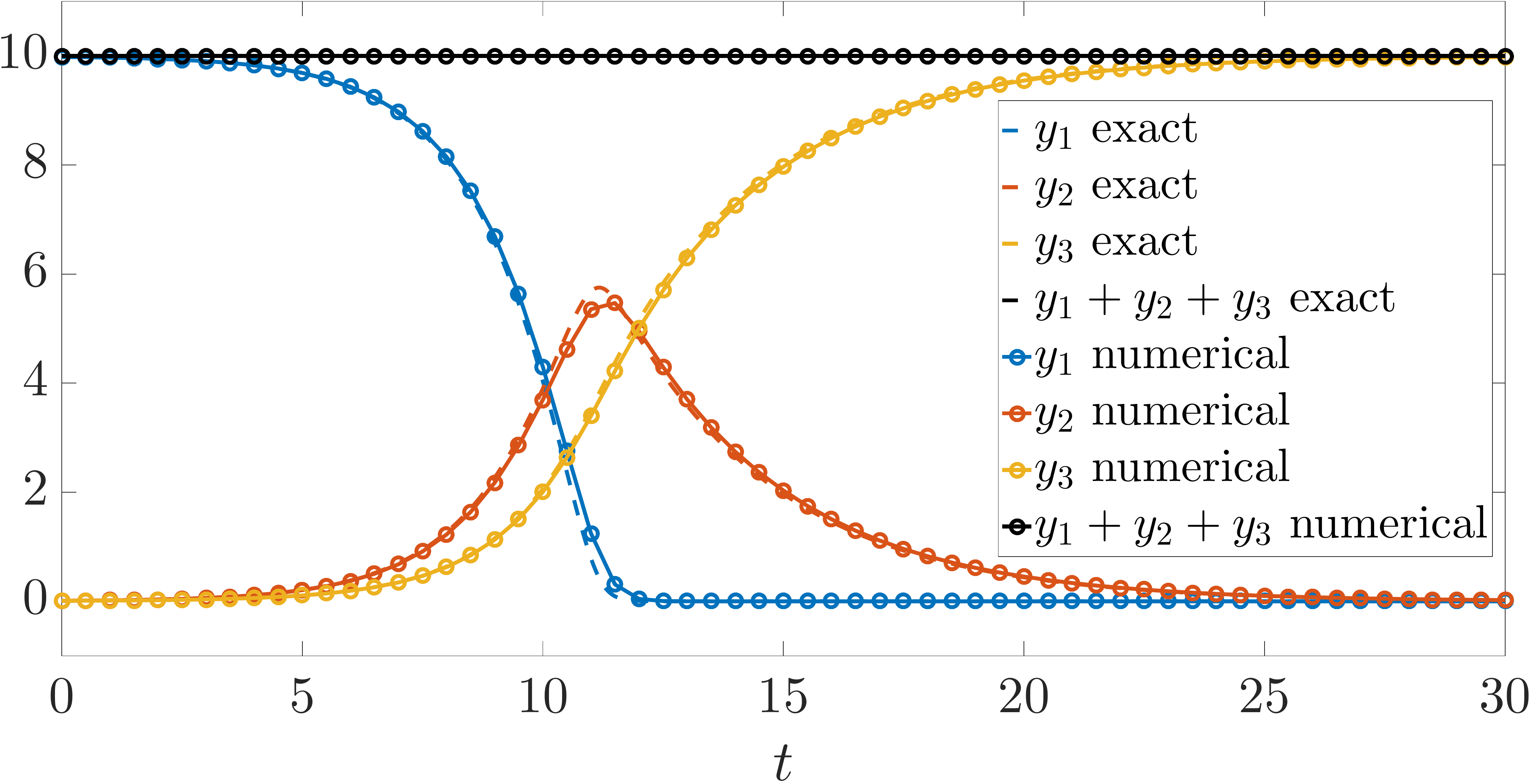}
 \caption{MPRK43I($1/2,3/4$)}
 \end{subfigure}
\hfill
 \begin{subfigure}{.49\textwidth}
\centering
 \includegraphics[width=\textwidth]{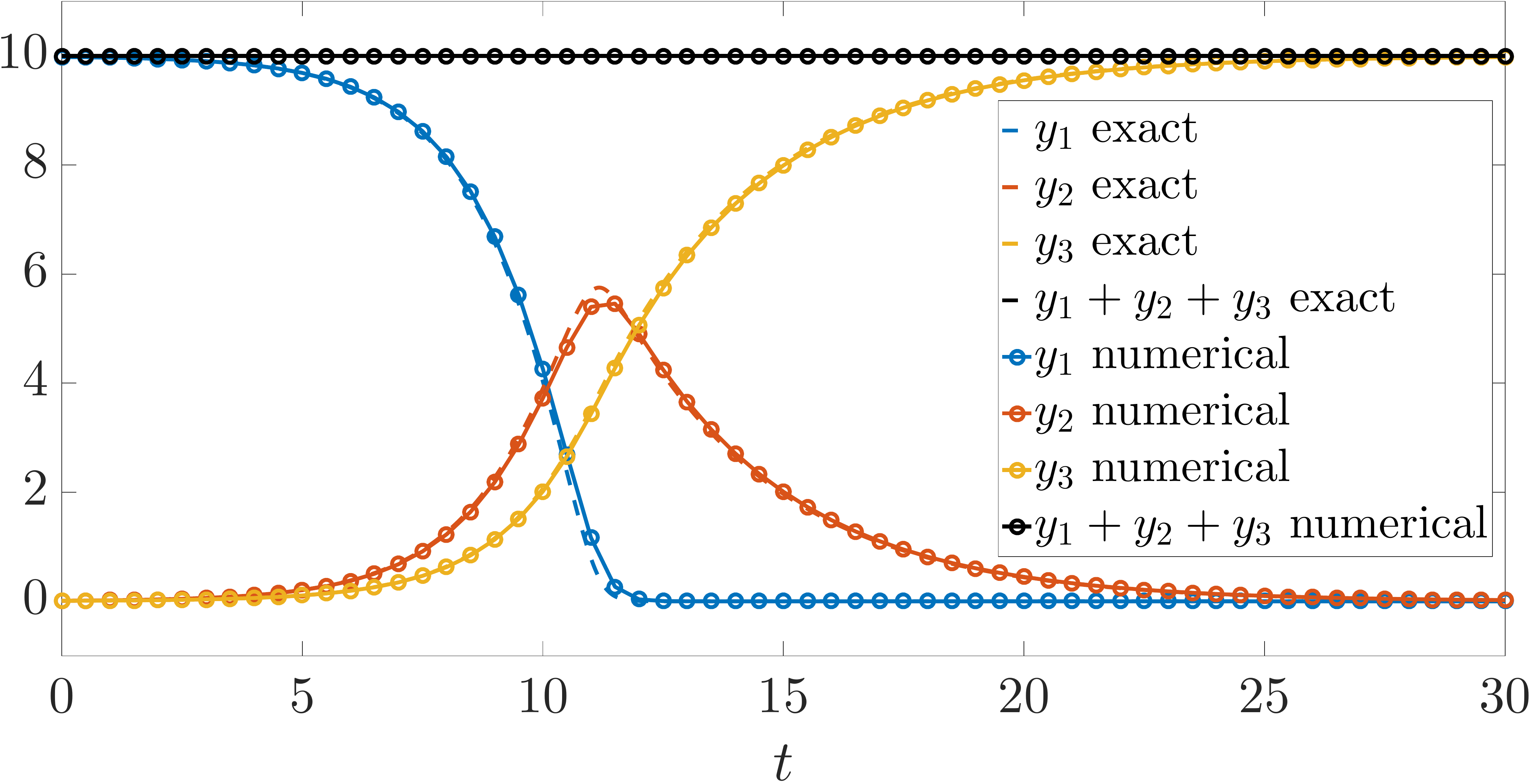}
 \caption{MPRK43Incs($1/2,3/4$)}
 \end{subfigure}
  \par
 \vspace{2\baselineskip}
  \begin{subfigure}{.49\textwidth}
\centering
 \includegraphics[width=\textwidth]{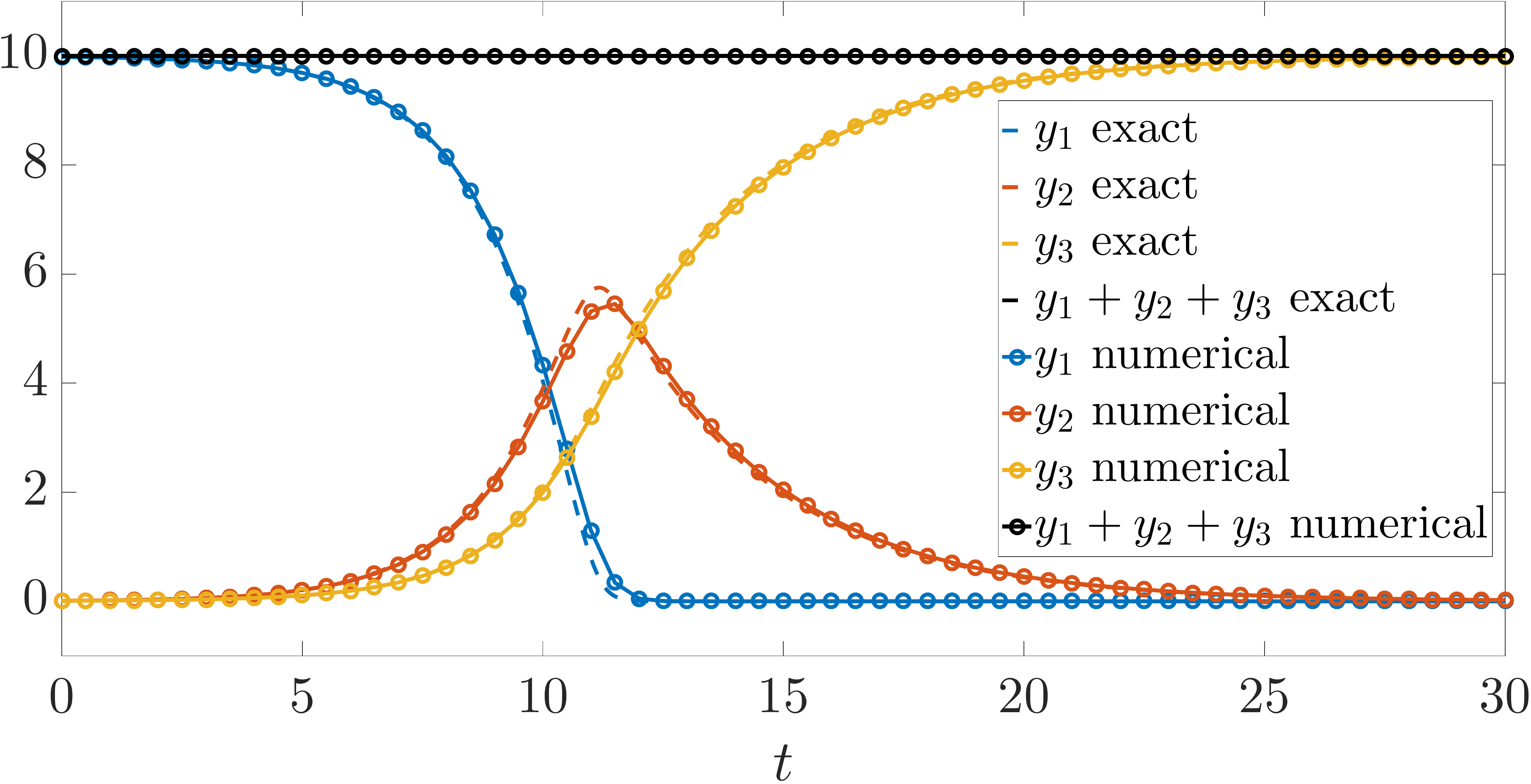}
 \caption{MPRK43II(1/2)}
 \end{subfigure}
 \hfill
  \begin{subfigure}{.49\textwidth}
\centering
 \includegraphics[width=\textwidth]{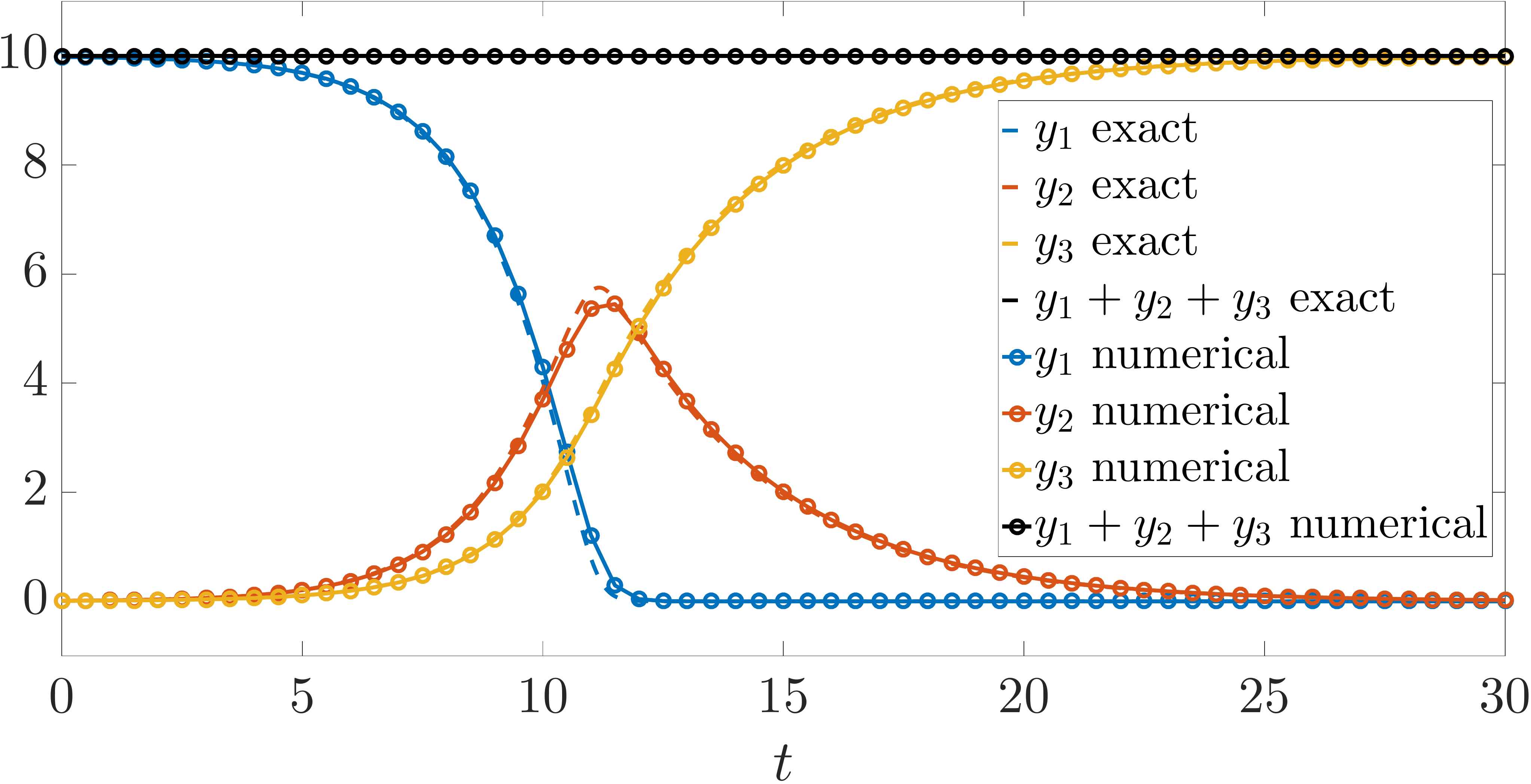}
 \caption{MPRK43IIncs(1/2)}
 \end{subfigure}
 \caption{Numerical solutions of the nonlinear test problem \eqref{eq:nonlintest} for different MPRK43 schemes with time step size $\Delta t=0.5$.}
 \label{fig:nonlin}
\end{figure}

\begin{figure}[htb]
\centering
\begin{subfigure}{.49\textwidth}
\centering
 \includegraphics[width=\textwidth]{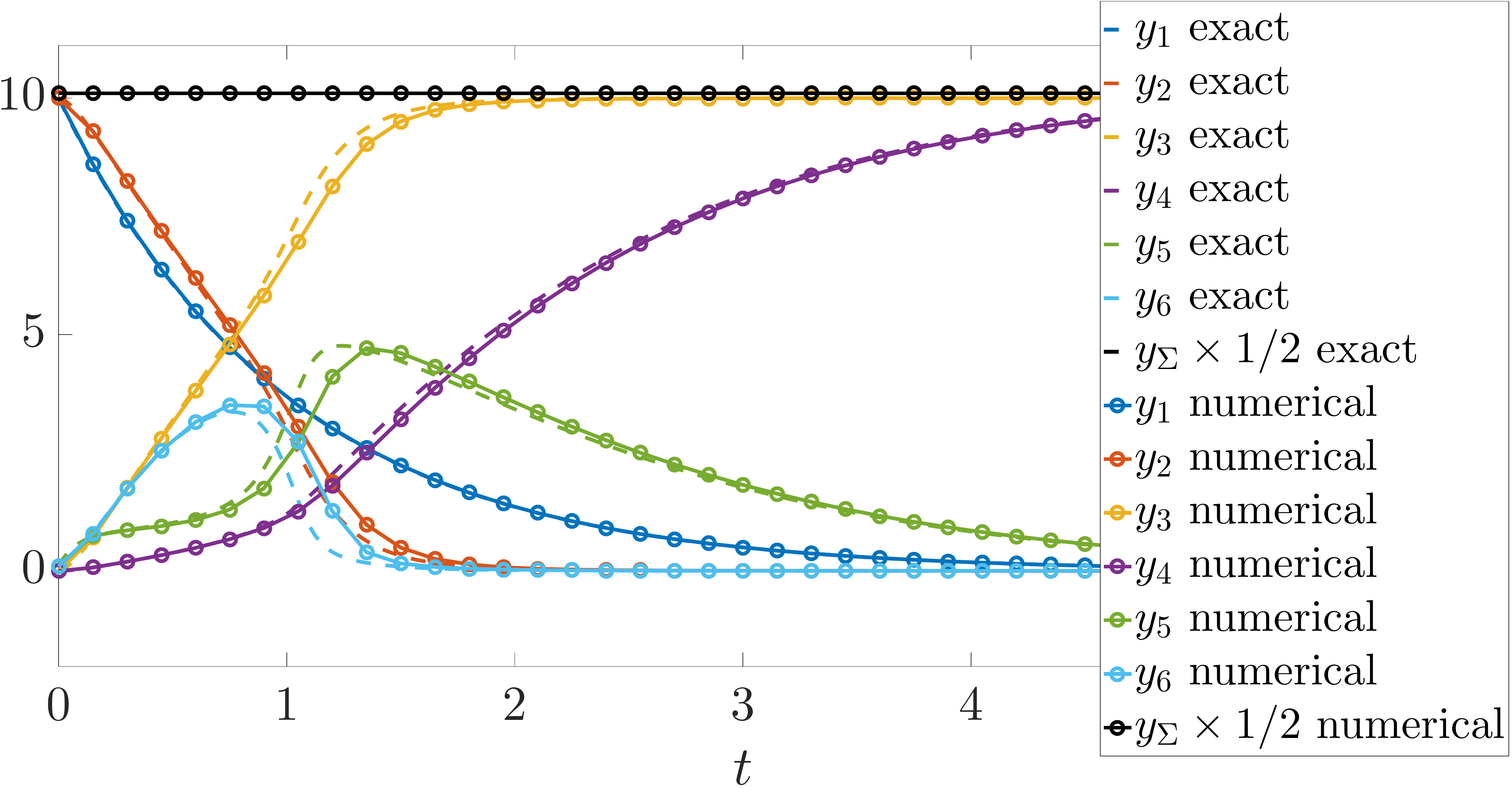}
 \caption{MPRK43I($1,1/2$)}
 \end{subfigure}
 \hfill
\begin{subfigure}{.49\textwidth}
\centering
 \includegraphics[width=\textwidth]{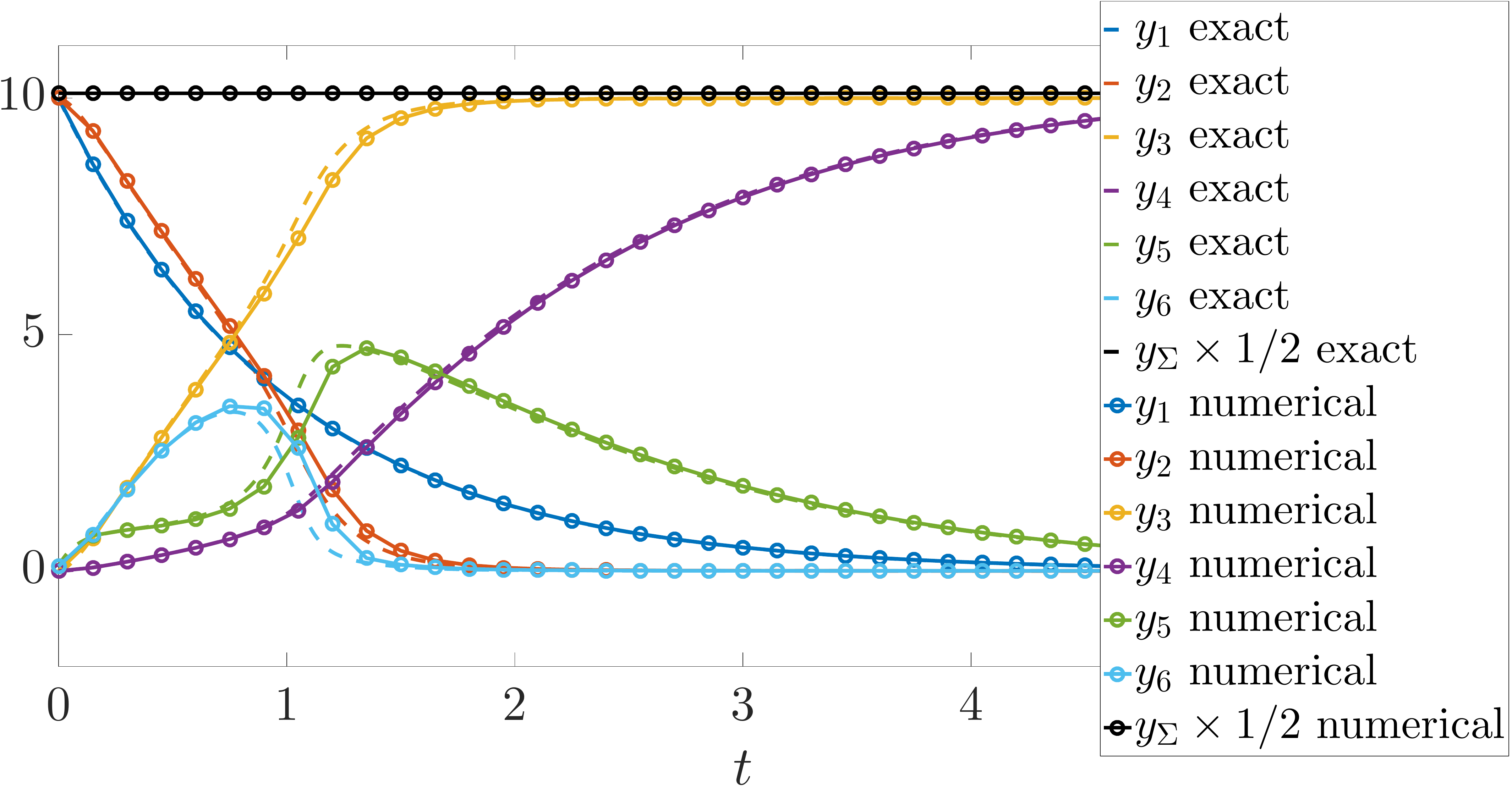}
 \caption{MPRK43Incs($1,1/2$)}
 \end{subfigure} 
 \par
\vspace{2\baselineskip} 
 \begin{subfigure}{.49\textwidth}
\centering
 \includegraphics[width=\textwidth]{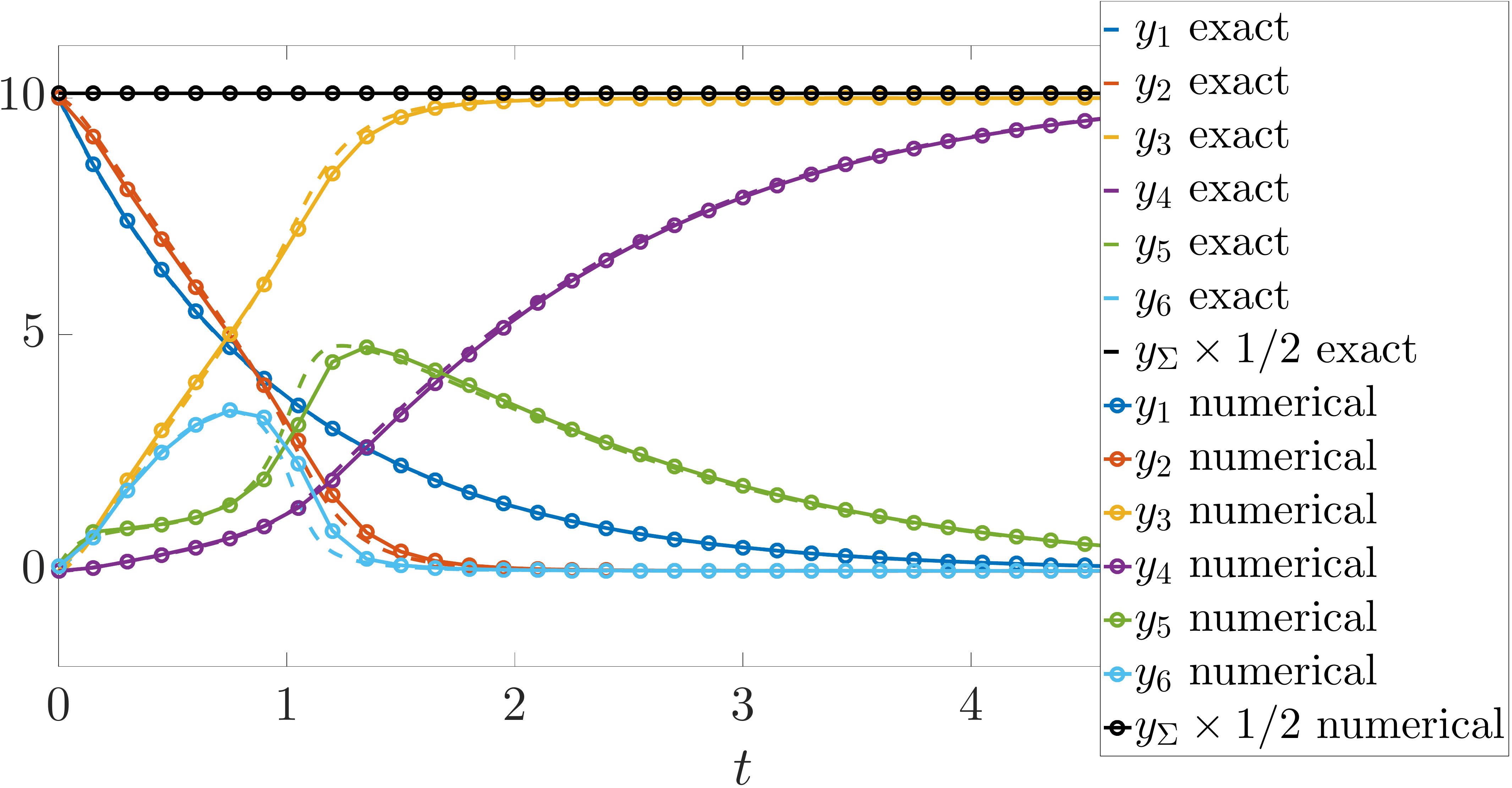}
 \caption{MPRK43I($1/2,3/4$)}
 \end{subfigure}
\hfill
 \begin{subfigure}{.49\textwidth}
\centering
 \includegraphics[width=\textwidth]{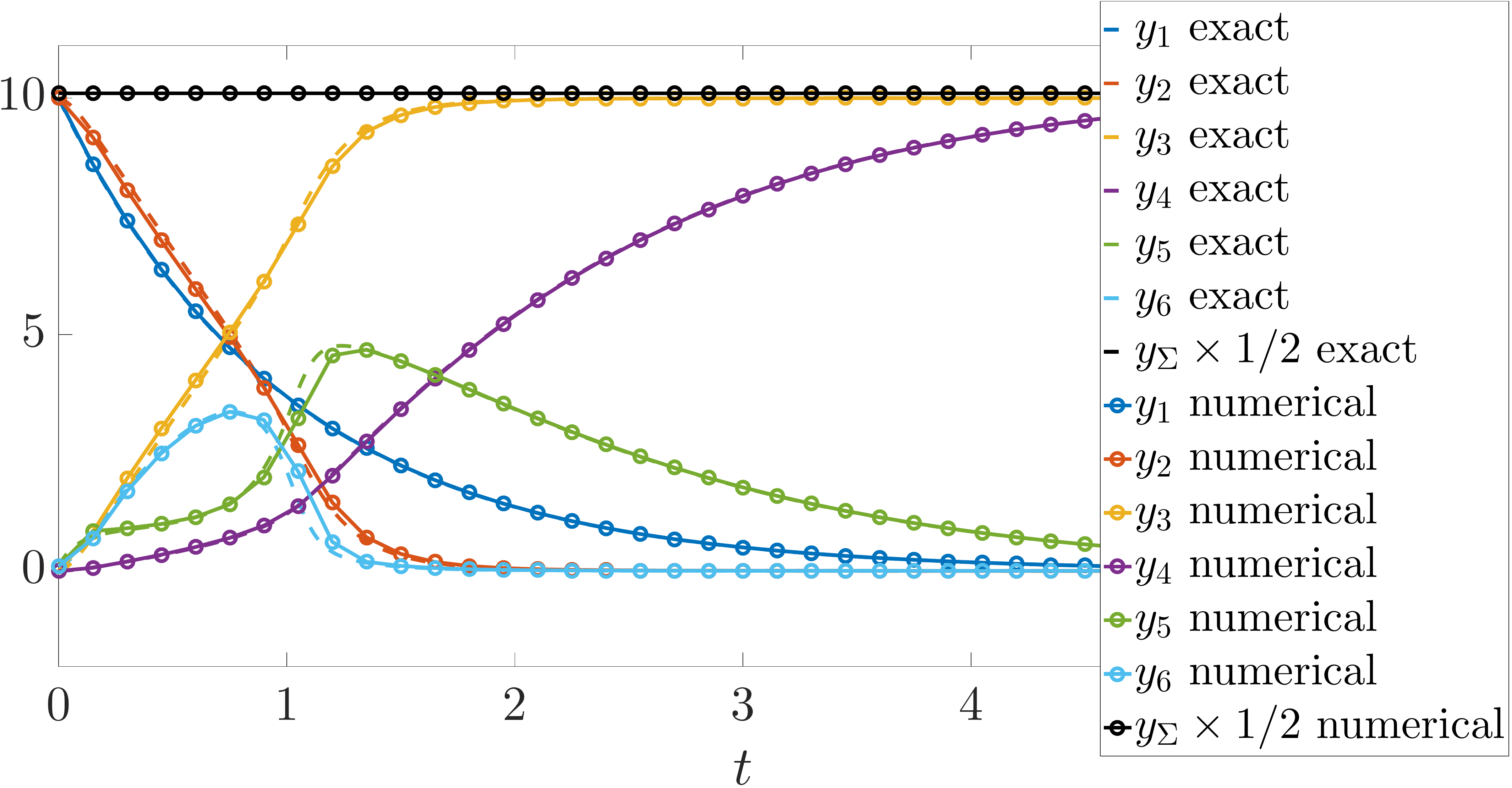}
 \caption{MPRK43Incs($1/2,3/4$)}
 \end{subfigure}
  \par
 \vspace{2\baselineskip}
  \begin{subfigure}{.49\textwidth}
\centering
 \includegraphics[width=\textwidth]{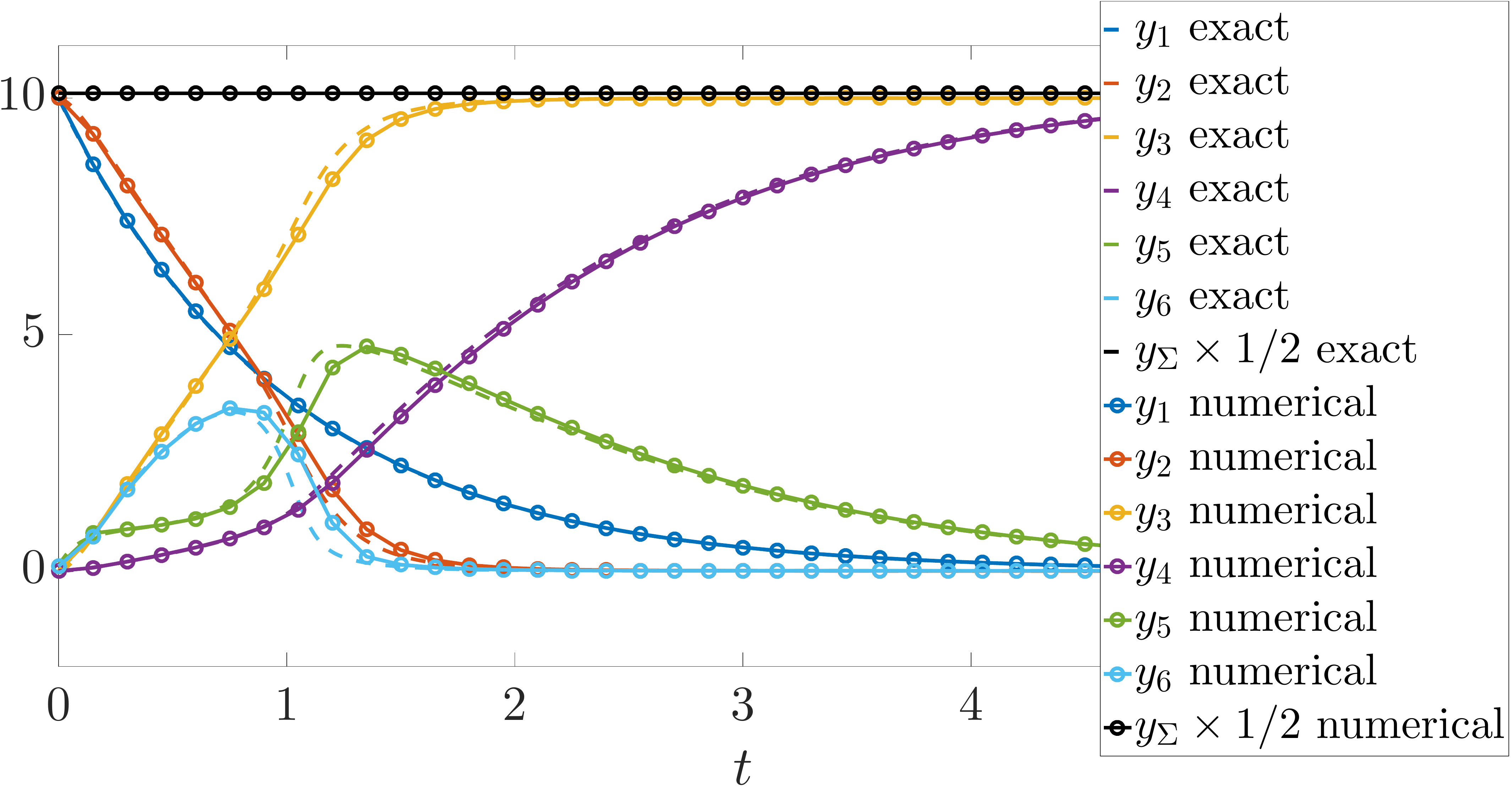}
 \caption{MPRK43II(1/2)}
 \end{subfigure}
 \hfill
  \begin{subfigure}{.49\textwidth}
\centering
 \includegraphics[width=\textwidth]{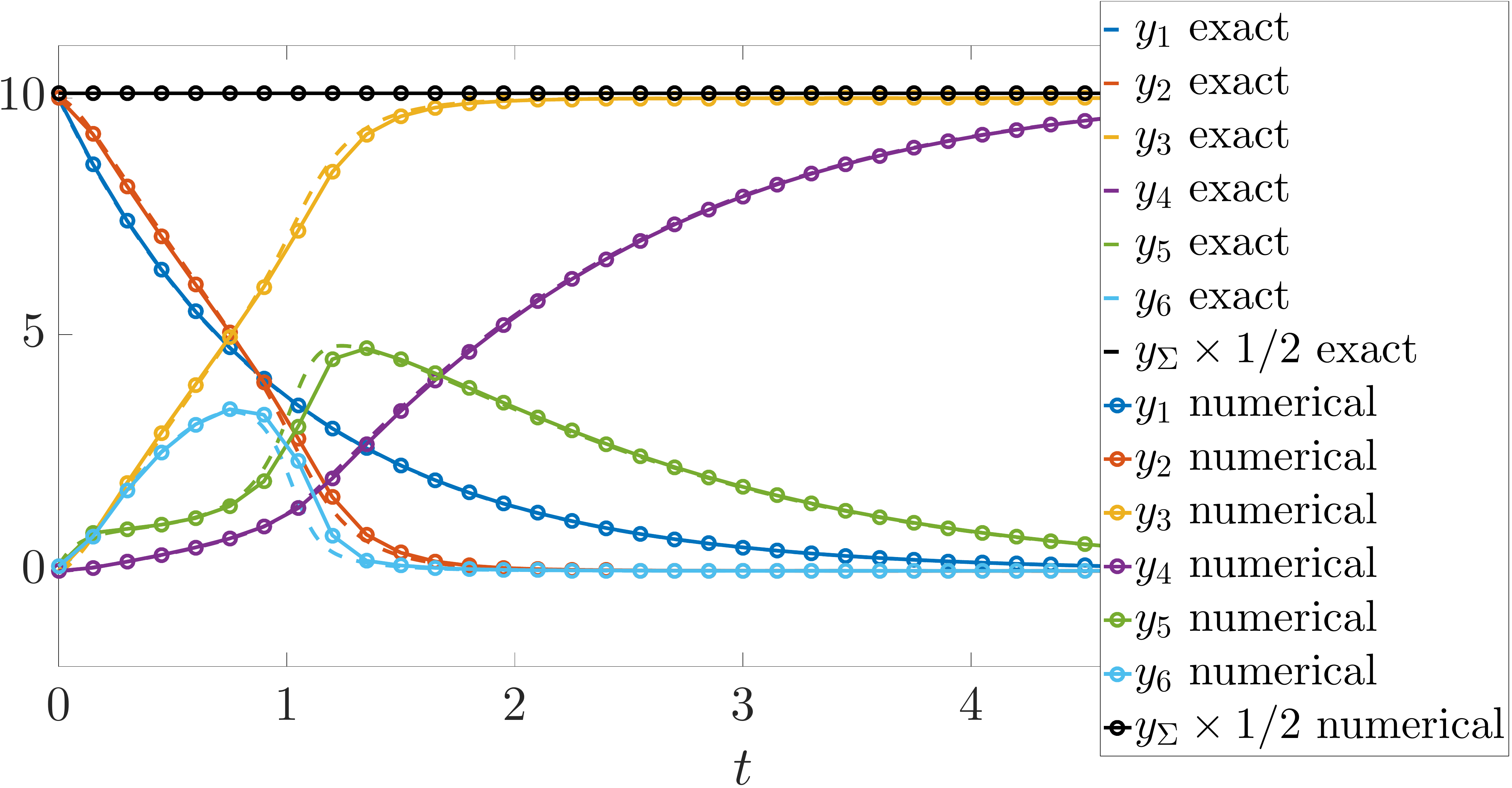}
 \caption{MPRK43IIncs(1/2)}
 \end{subfigure}
 \caption{Numerical solutions of the Brusselator problem \eqref{eq:brusselator} for different MPRK43 schemes with time step size $\Delta t=0.15$. The term $y_\Sigma$, which appears in the legend, is defined as $y_\Sigma=y_1+\dots+y_6$.}
 \label{fig:brusselator}
\end{figure}
In this section, we confirm the theoretical convergence order of the novel MPRK43 schemes. 
We also show numerical approximations of MPRK43 schemes applied to the stiff Robertson problem \eqref{eq:robertson}, the nonlinear test problem \eqref{eq:nonlintest} and the Brusselator problem \eqref{eq:brusselator}. 

To visualize the order of the MPRK schemes we use a relative error $E$ taken over all time steps and all constituents:
\begin{equation*}
 E = \frac{1}{N}\sum_{i=1}^N E_i,\quad E_i = \Bigl(\frac{1}{M}\sum_{m=1}^M y_i(t^m)\Bigr)^{\!\!-1}\Bigl(\frac{1}{M}\sum_{m=1}^M\left(y_i(t^m) - y_i^m\right)^2\Bigr)^{\!\!1/2},
\end{equation*}
where $M$ denotes the number of executed time steps. 
To compute the error $E$ we need to know the analytic solution, which is known for the linear test case, but not for the other test problems.
Hence, we computed a reference solution, using the \textsc{Matlab} functions \texttt{ode45} for the non-stiff nonlinear problems and \texttt{ode23s} for the Robertson problem. 
In both cases we utilized the tolerances $\mathtt{AbsTol}=\mathtt{RelTol}=10^{-10}$.

\subsection*{Convergence order}
Figure~\ref{fig:testorder} shows error plots of six MPRK43 schemes applied to the linear test problem \eqref{eq:lintest}, the nonlinear test problem \eqref{eq:nonlintest} and the Brusselator \eqref{eq:brusselator}.
In all cases the third order accuracy is confirmed.
Moreover, Figure~\ref{fig:testorder1} shows that MPRK43I(1,1/2) is less accurate than MPRK43IIncs(1,1/2) and MPRK43I(1/2,3/4) is more accurate than MPRK43Incs(1/2), when applied to the linear test problem. 
Hence, we cannot make a general statement, whether to choose $\delta=0$ or $\delta=1$ in the MPRK43 schemes.

\subsection*{Stiff problems}
Figure~\ref{fig:robertson} shows numerical approximations of six MPRK43 schemes applied to the stiff Robertson problem \eqref{eq:robertson}. As mentioned, the time step size in the $k$th time step was chosen as $\Delta t_k=4^{k-1}\Delta t_0$ with initial time step size $\Delta t_0=10^{-6}$.
Hence, only 29 time steps are necessary to traverse the time interval $[10^{-6},10^{10}]$.
The small initial time step was chosen to obtain an adequate resolution of the component $y_2$ in the starting phase. 
To visualize the evolution of $y_2$, it is multiplied by $10^4$.

All six schemes generate adequate solutions. 
For this test problem, the variants with conservative stage values (left column) can be seen to be more accurate than those with non-conservative stage values (right column). 
But the overall accuracy is excellent with regard to the fairly large time steps in use.

In \cite{KopeczMeister2017} we reported that some MPRK22ncs schemes generate oscillations, when applied to the Robertson problem. 
In case of the MPRK43 schemes, we did not encounter any issues.

Like in \cite{KopeczMeister2017}, we additionally show numerical solutions of the six MPRK43 schemes applied to the nonlinear test problem \eqref{eq:nonlintest} and the Brusselator \eqref{eq:brusselator} in Figures~\ref{fig:nonlin} and \ref{fig:brusselator}.
\section{Summary and Outlook}
In this paper we have extended the work of \cite{KopeczMeister2017} to third order by deriving necessary and sufficient conditions for three-stage third order schemes. 
We also introduced the MPRK43I and MPRK43II schemes, which, to our knowledge, are the first third order Patankar-type schemes presented in literature. 
These schemes can be regarded as four-stage third order MPRK schemes and it is a future research topic, to investigate the construction of three-stage third order MPRK schemes.
In addition, the search for other possible PWDs is of interest as well.

The numerical experiments have shown that the MPRK43 schemes are capable of integrating stiff ODEs, such as the Robertson problem.
However, in absence of a thorough analysis of truncation errors and stability, we cannot make statements which schemes of the MPRK43 family are preferable. 
This is a future research topic as well.
\clearpage
\newcommand{\etalchar}[1]{$^{#1}$}

\end{document}